\documentclass{amsart}
 \usepackage{amsmath}
\usepackage{amssymb}
\usepackage{amsfonts}

\newtheorem{theorem}{Theorem}[section]
\newtheorem{lemma}[theorem]{Lemma} \newtheorem{corollary}[theorem]{Corollary}
\newtheorem{proposition}[theorem]{Proposition}
\newtheorem{remark}[theorem]{Remark}

\numberwithin{equation}{section}
\usepackage{a4wide}
\def \a{{\alpha}}
\def \b{{\beta}}

\def \d{{\delta}}
\def \e{{\varepsilon}}
\def \g{{\gamma}}

\def \o{{\omega}}
\def \O{{\Omega}}
\def \p{{\varphi}}

\def \m{{\mu}}
\def \s{{\sigma}}

\def \dd{{\rm d}}

\def \qq{{\qquad}}
\def \noi{{\noindent}}

\def\E{{\mathbb E \,}}

\def\R{{\mathbb R}}

\def\N{{\mathbb N}}

\scrollmode

\font\cmit=cmti9 at 10,6 pt

\font\sevenrm= cmr10 at 8 pt

\font\rm= cmr10 at 10,6 pt  \scrollmode
 \font\cmssqi=cmssqi8 at 9,5 pt

\def\ddate {\sevenrm \ifcase\month\or January\or
February\or March\or April\or May\or June\or July\or
August\or September\or October\or November\or December\fi\! {\the\day}, \!{\sevenrm\the\year}}

 \title[An extension of  a  result of Erd\H os and Zaremba]
 {An extension of  a  result of Erd\H os and Zaremba}
   
\begin{document}
 \author{Michel J.\,G.  WEBER}
 \address{IRMA, UMR 7501,   7  rue  Ren\'e Descartes, 67084 Strasbourg Cedex, France.}
\email{michel.weber@math.unistra.fr}




  \begin{abstract}
  Erd\"os and  Zaremba  showed that
$ \limsup_{n\to \infty} \frac{\Phi(n)}{(\log\log n)^2}=e^\g$, $\g$ being Euler's constant, where $\Phi(n)=\sum_{d|n} \frac{\log d}{d}$. 

We  extend this result to the function $\Psi(n)= \sum_{d|n} \frac{(\log d )(\log\log d)}{d}$ and some other  functions.   We show
that
$    \limsup_{n\to \infty}\, \frac{\Psi(n)}{(\log\log n)^2(\log\log\log n)}\,=\, e^\g$.
 The proof
 requires  a new approach. As  an application, we prove  that    for any    $\eta>1$, any finite sequence of reals $\{c_k, k\in K\}$,
$\sum_{k,\ell\in K} c_kc_\ell \, \frac{\gcd(k,\ell)^{2}}{k\ell} \le  C(\eta) \sum_{\nu\in K} c_\nu^2(\log\log\log \nu)^\eta \Psi(\nu)
$,
where $C(\eta)$ depends on $\eta$ only. This
 improves  a recent result obtained by the author.

  \end{abstract}

\maketitle

 \section{\bf Introduction.} \label{s1}
  \rm 
Erd\"os and  Zaremba   showed in  \cite{EZ} the following result concerning  the arithmetical function $\Phi(n)=\sum_{d|n} \frac{\log d}{d}$,
 \begin{equation}\label{EZ1}\limsup_{n\to \infty} \frac{\Phi(n)}{(\log\log n)^2}=e^\g,
\end{equation}
  where  $\g $ is Euler's constant. This function appears in the study of good lattice points in numerical integration, see Zaremba \cite{Z}. The proof
  is based on the identity
\begin{eqnarray}\label{formule}\Phi(n)
=\sum_{i=1}^r \sum_{\nu_i=1}^{\a_i}\frac{\log p_i^{\nu_i}}{p_i^{\nu_i}}\sum_{\d|n p_i^{-\a_i}}\frac{1}{\d} , \qq\qq ( n= p_1^{\a_1}\ldots p_r^{\a_r}),
\end {eqnarray}
which follows from
\begin{eqnarray}\label{base} \sum_{d|n} \frac{\log d}{d}&=& \sum_{\m_1=0}^{\a_1}\ldots \sum_{\m_r=0}^{\a_r} \frac{1}{ p_1^{\m_1}\ldots p_{r}^{\m_{r}}}\Big(\sum_{i=1}^{r}\m_i\log p_i\Big) .
\end{eqnarray}
 Let $h(n)$ be non-decreasing on integers, $h(n)= o(\log n)$, and
  consider the slightly larger function
\begin{eqnarray}\label{phi}\Phi_h(n)=\sum_{d|n} \frac{(\log d )\,h(d)}{d^{}}.
\end{eqnarray}

In this case  a  formula similar to \eqref{base}   no longer holds, the "log-linearity" being  lost due to the extra factor $h(n)$. The study  of this function requires   a new approach.
 We study in this work the case $h(n)= \log\log n$, that is the function
\begin{eqnarray}\label{psi}\Psi (n)=\sum_{d|n} \frac{(\log d )(\log\log d)}{d^{}}.
\end{eqnarray}
 We extend Erd\H os-Zaremba's result for this function, as well as for the  functions
\begin{eqnarray*} \Phi_1(n)&=&\sum_{
p_1^{\m_1}\ldots p^{\m_{r}}_{r}|n}\frac{\sum_{i=1}^{r} \m_i(\log p_i)(\log\log p_i)}{p_1^{\m_1}\ldots p^{\m_{r}}_{r}}
\cr \Phi_2(n)&=&\sum_{d|n}\frac{(\log d)
\log\, \O(d) }{d},
\end{eqnarray*}
where $\O(d)$ denotes as usual   the total number of  prime factors of $d$ counting multiplicity. These    functions are linked to $\Psi$.
\vskip 2 pt
Throughout,    $\log\log x$  (resp. $\log \log\log x$) equals  $1$ if $0\le x \le e^{e}$ (resp. $0\le x \le e^{e^e}$), and  equals $\log\log x$ (resp. $\log \log\log x$) in the usual sense if $x> e^{e}$ (resp. $x> e^{e^e}$).
\vskip 3pt

One  verifies using standard arguments that
\begin{equation}  \label{phipsi}  \limsup_{n\to \infty}\, \frac{\Phi_1(n)}{(\log \log n)^2(\log\log\log n)}\,\ge \,e^\g, \qq
 \limsup_{n\to \infty}\, \frac{\Psi(n)}{(\log \log n)^2(\log\log\log n)}\,\ge \,e^\g,\end{equation}
and  in fact that
 \begin{equation}   \label{Phi1est}  \limsup_{n\to \infty}\, \frac{\Phi_1(n)}{(\log \log n)^2(\log\log\log n)}\,= \,e^\g .
\end{equation}

\vskip 2 pt
By the observation made after \eqref{base},
the corresponding extension of this result to $\Psi(n)$
is  technically more delicate. It follows from \eqref{EZ1} that
 \begin{equation}\label{trois} \limsup_{n\to \infty}\, \frac{\Psi(n)}{(\log \log n)^3}\, \le \, e^\g.
  \end{equation}
The   question thus arises  whether the exponent of   $\log\log n$  in \eqref{trois} can be replaced by $2+\e$, with $\e>0$ small.

\vskip 2 pt We  answer this question affirmatively by establishing the following precise result, which is the main result of this paper.

\begin{theorem}\label{t1}
\begin{equation*}
 \limsup_{n\to \infty}\, \frac{\Psi(n)}{(\log \log n)^2(\log\log\log n)}\, =\, e^\g.\end{equation*}
\end{theorem}

An application of this result is given in Section \ref{s6}.
  The upper bound  is obtained, via the    inequality
  \begin{eqnarray}\label{convexdec} \Psi(n)\ \le \  \Phi_1(n) + \Phi_2(n),\end{eqnarray}
 as a combination of an estimate of   $\Phi_1(n)$ and  the following estimate of  $\Phi_2(n)$.  Recall that Davenport's function $w(n)$ is defined by  $w(n)=\sum_{p|n}\frac{\log p}{p}$.
According to Theorem  4 in  \cite{D} we have,
 \begin{equation}\label{wdavenport}\limsup_{n\to \infty}\frac{w(n)}{\log\log n}=1.
 \end{equation}

  Let also $\o(n)$ be  the number of prime divisors of $n$ counted without multiplicity.
\begin{theorem}\label{t2}For all   odd numbers $n$ we have,
\begin{eqnarray*}
\Phi_2(n)\, \le \, C\,  (\log\log\log \o(n))(\log \o(n))w(n)
.\end{eqnarray*}
where   $C$ is an absolute constant.\end{theorem}

Here and elsewhere $C$  (resp. $C(\eta)$) denotes some positive absolute constant (resp. some  positive constant depending only of a parameter  $\eta$).
\vskip 2 pt
The approach used for proving  Theorem \ref{t2} can be adapted with no  difficulty to other  arithmetical  functions of similar type.

 \vskip 5 pt  Before continuing we mention some other existing extensions,
 due to  Sitaramaiah and Subbarao in \cite{SS2,SS1}. For instance, the case when $\log d$ is replaced by a non-negative additive function $g$ (ie. $S(n)=\sum_{d|n} \frac{g(d)}{d}$) is studied in \cite{SS1}. In our case we note that $g(d)= (\log d)(\log\log d)$ (see \eqref{phi}), which is obviously not additive.   It is proved that if $T(d)$ is one of the three arithmetical functions $\frac{\o(d)}{d}$, $\frac{\O(d)}{d}$, $\frac{\log \tau(d)}{d}$, then
\begin{equation}\label{sisu1}\limsup_{n\to \infty} \frac{\sum_{d|n}T(d)}{(\log\log n)(\log\log \log n)}=c_T\,e^\g,
\end{equation}
where $c_T= 1$ in the two first cases, and  $c_T= \log 2$ in the third case.  See also Remark 2.3. A basis of their proof lies in the observation that $S(n)/\s_{-1}(n)$ is additive. Further it is   proved in \cite{SS2} that
 for each positive integer $k$,
 \begin{equation}\label{sisu2}\limsup_{n\to \infty} \frac{S_k(n)}{(\log\log n)^{k+1}} =c_k\,e^\g,\qq \qq (S_k(n)=\sum_{d|n}\frac{(\log d)^k}{d})
\end{equation}
where $c_k$ is a positive explicit constant. The proof is elegant and based on the  derivation formula $S_k(n)=(-1)^kf^{(k)}(1)$, where $f(u)= \s_{-u}(n)$ and $f^{(k)}$ is the $k$-th derivative of $f$, which is specific to these sums.
 \vskip 4 pt  The paper is organized as follows.
 Sections \ref{s4} and \ref{s5} form the main part of the paper, and consist of the proof of Theorem \ref{t2}, which is long and technical and involves the building of a binary tree (subsection  \ref{subsection4.2.1}). The proof of Theorem \ref{t1} is given in section \ref{s5}.
 Section \ref{s2} contains complementary results and the proofs of \eqref{phipsi}, \eqref{Phi1est}. Section \ref{s6} concerns the afore mentioned application of Theorem \ref{t1}.  Additional remarks and results are given  in Section \ref{s7}.

\vskip 2 pt \hskip -3 pt

 \section{\bf Proof of Theorem \ref{t2}.} \label{s4}

We use a chaining argument.     We make throughout  the convention   $0\log 0=0$.
\rm
\vskip 6 pt
Let $n= p_1^{\a_1}\ldots p_r^{\a_r}$ be an odd number. We will use repeatedly the fact that
 \begin{eqnarray}\label{min}\min_{i=1}^r p_i\ge 3.
\end{eqnarray}
 \vskip 10 pt

 We  note
  that
 \begin{eqnarray}\label{basic}
  \Phi_2(n)&=&\sum_{\m_1=0}^{\a_1}\ldots \sum_{\m_r=0}^{\a_r} \frac{1}{ p_1^{\m_1}\ldots p_{r}^{\m_{r}}}\sum_{i=1}^{r} \m_i\big(\log p_i\big)\log \Big(\sum_{i=1}^{r} \m_i \Big)
\cr &=& \sum_{i=1}^r\underbrace{\sum_{\m_1=0}^{\a_1}\ldots \sum_{\m_r=0}^{\a_r}}_{\substack{ \hbox{\small  the sum relative}\\ \hbox{\small to $ \m_i$ is excluded}  } } \underbrace{\frac{1}{ p_1^{\m_1}\ldots p_{r}^{\m_{r}}}}_{\substack{\hbox{\small  $ p_i^{\m_i}$}\\ \hbox{\small  is excluded}}}\ \Big(\sum_{\m_i=0}^{\a_i} \frac{\m_i\big(\log p_i\big)}{p_i^{\m_i}}\Big)\log \Big(\sum_{i=1}^{r} \m_i \Big)
.
\end{eqnarray}
   
As there is no order relation on the sequence $p_1, \ldots, p_r$, it suffices to  study   the sum
 \begin{eqnarray} \label{Phi_2(r,n)}
\Phi_2(r,n)&:=& \sum_{\m_1=0}^{\a_1}\ldots \sum_{\m_{r-1}=0}^{\a_{r-1}} \frac{1}{ p_1^{\m_1}\ldots p_{r-1}^{\m_{r-1}}}\sum_{\m_r=0}^{\a_r}\frac{\m_r \log p_r}{p_r^{\m_r}}\log \Big[\sum_{i=1}^{r-1} \m_i + \m_r\Big].\end{eqnarray}
The  sub-sums in \eqref{Phi_2(r,n)} will be estimated by using a recursion argument.

\vskip 3 pt
We first explain its principle and examine the structure of the sum $\Phi_2(r,n)$, anticipating somehow the calculations.
The last sum
$\sum_{\m_r=0}^{\a_r}\frac{\m_r \log p_r}{p_r^{\m_r}}\log \big[\sum_{i=1}^{r-1} \m_i + \m_r\big]$ is of  type
$$\sum_{\m=0}^{\a_r}\a \m \big(\log (A+\m)\big)e^{-\a \m}, \qq \qq \hbox{$\a =\log p_r$, \quad $A=\sum_{i=1}^{r-1}\m_i$}.$$
It is easy to observe  with \eqref{49} that the bound obtained in Lemma \ref{E1},   will induce on the sum in  $\m_{r-1}$ a logarithmic factor $\log \big[h+ \sum_{i=1}^{r-2} \m_i + \m_{r-1}\big]$ where $h$ is a positive integer, and  so one. More precisely,
  \begin{eqnarray*}
 & & \sum_{\m=0}^{\infty}\a \m \big(\log (A+\m)\big)e^{-\a \m}\ \le \  \a  \big(\log (A+1)\big)e^{-\a}+ 2\a  \big(\log (A+2)\big)e^{-2\a}
  \cr & &\qq   + \Big\{3\a \log (A+3)+3 \log (A+3)+  \frac{1}{\a} \log (A+3) +  \frac{1}{\a}
+ \frac{1}{\a^2(A+3)}\Big\} e^{-3\a}.\end{eqnarray*}
provided $A\ge 1$, and $\a \ge 1$. Whence the bound,
    \begin{eqnarray*}
& &  \frac{\log p_r}{p_r}  \big(\log (A+1)\big)+    \frac{2\log p_r}{p^2_r}\big(\log (A+2)\big)
     + \frac{1}{p^3_r} \Big\{3\log p_r \log (A+3)+3 \log (A+3)\cr & &+  \frac{1}{\log p_r} \log (A+3) +  \frac{1}{\log p_r}
+ \frac{1}{(\log p_r)^2(A+3)}\Big\}.\end{eqnarray*}
By reporting this bound in \eqref{Phi_2(r,n)}, we get 
sums of type
\begin{eqnarray*}
 \sum_{\m_1=0}^{\a_1}\ldots \sum_{\m_{r-1}=0}^{\a_{r-1}} \frac{\log \big[\sum_{i=1}^{r-1} \m_i + h\big]}{ p_1^{\m_1}\ldots p_{r-1}^{\m_{r-1}}} \qq h=1,2,3,\end{eqnarray*}
affected with new coefficients, this is displayed in   \eqref{h1234}.   By   using \eqref{sumphi2}, the last sum is bounded by
\begin{eqnarray*} & &  \log \big[\sum_{i=1}^{r-2} \m_i + h\big]
  + \frac{1}{p_{r-1}}\Big(\log (\sum_{i=1}^{r-2} \m_i + h+1)\cr & &+ \frac{\log\big[\sum_{i=1}^{r-2} \m_i + h+1\big]}{\log p_{r-1}} + \frac{1}{(\log p_{r-1})^2(\sum_{i=1}^{r-2} \m_i + h+1)}\Big) .\cr{} & &
\end{eqnarray*}
By recursing once more, this allows one  to bound again    $\Phi_2(r,n)$. The remainding sums will after be all of same type. The factor $ \log \big[\sum_{i=1}^{r-1} \m_i + h\big]$ induces on the sum of order $(r-2)$ a factor  $\log\big[\sum_{i=1}^{r-2} \m_i + h+1\big]$.
The whole matter thus consists with understanding how the new coefficients are generated, and in particular to check whether a coefficient of order  $1+\e$ will not produce by iteration a coefficient of order $(1+\e)^r$. A recurrence inequality  established in Lemma \ref{E3} will allow one to  control their magnitude efficiently.


\subsection{Preparation} Some technical  lemmas will be needed.


\begin{lemma}\label{phivar} {\rm (i)}  Let $\p_1(x)=x\big(\log (A+x)\big) e^{-\a x}$, $\p_2(x)=\big(\log (A+x)\big) e^{-\a x}$.   Then $\p_1(x)$ is non-increasing on  $[3, \infty)$ if $A\ge 1$ and $\a\ge \log 2$. Further, $\p_2(x)$ is non-increasing on $[1,\infty)$, if  $A\ge 1$ and $\a\ge 1$.

\vskip 3 pt \noi {\rm (ii)} Assume that  $A\ge 1$ and $\a\ge \log 2$. For any integer $m\ge 1$,
\begin{eqnarray} \label{phi.intest1}
\a \int_m^\infty x \big(\log (A+x)\big) e^{-\a x}\hbox{\rm  d}x&\le & \frac{1}{\a^2(A+m)}e^{-\a m}+  \frac{1}{\a}e^{-\a m}  +  \frac{1}{\a}\big(\log (A+m)\big)e^{-\a m}
\cr & &  + m(\log A+m) e^{-\a m}.\end{eqnarray}

\vskip 3 pt \noi  {\rm (iii)} Assume that $A\ge 1$ and $\a\ge 1$. Then,
  \begin{eqnarray}\label{intest2}
  \int_1^\infty \big(\log (A+x)\big) e^{-\a x}\hbox{\rm  d} x&\le &\frac{\log (A+1)}{\a} e^{-\a}+\frac{1}{\a^2(A+1)}e^{-\a}.
  \end{eqnarray}
\end{lemma}

\begin{proof}[\cmit Proof] \rm (i) We have
 $\p_1'(x) = \big(\log (A+x)\big)e^{-\a x}+ \frac{x}{A+x}e^{-\a x}
-\a x \big(\log (A+x)\big) e^{-\a x}$.
By assumption and since $\p_1'(x)\le 0 \Leftrightarrow  \frac{1}{x}+ \frac{1}{(A+x)\log (A+x)}\le \a$, we get
$$ \frac{1}{x}+ \frac{1}{(A+x)\log (A+x)}\le\frac{1}{3}+\frac{1}{8\log 2}  \le \frac{1}{3}+\frac{1}{5}<\log 2\le \a.$$
Similarly
$ \p_2'(x)=\frac{1}{A+x}e^{-\a x}-\a \big(\log (A+x)\big) e^{-\a x}$.
 As
$\p_2'(x)\le 0 \Leftrightarrow (A+x)\log (A+x)\ge \frac{1}{\a}$, we also get $$(A+x)\log (A+x)\ge 2\log 2>1\ge \frac{1}{\a}.$$


\noi (ii) \rm We deduce from (i) that
\begin{equation}\label{49} \a x \big(\log (A+x)\big) e^{-\a x}=  \big(\log (A+x)\big)e^{-\a x}+ \frac{x}{A+x}e^{-\a x}-\big( x (\log A+x) e^{-\a x}\big)' .
 \end{equation}
By integrating,
 \begin{eqnarray*}\label{phi1int}
\a \int_m^\infty x \big(\log (A+x)\big) e^{-\a x}\hbox{\rm  d}x&=&  \int_m^\infty x\big(\log (A+x)\big)e^{-\a x}\hbox{\rm  d} x+ \int_m^\infty \frac{x}{A+x}e^{-\a x}\hbox{\rm  d} x
\cr & & \quad + m(\log A+m) e^{-\a m}.\end{eqnarray*}
Similarly
 \begin{equation*}
  \a \int_m^\infty \big(\log (A+x)\big) e^{-\a x}\hbox{\rm  d} x=\int_m^\infty
 \frac{1}{A+x}e^{-\a x}\hbox{\rm  d} x+ \big(\log (A+m)\big) e^{-\a m}.\end{equation*}
By combining we get,
\begin{eqnarray*}
\a \int_m^\infty x \big(\log (A+x)\big) e^{-\a x}\hbox{\rm  d}x&=& \frac{1}{\a}\int_m^\infty\frac{1}{A+x}e^{-\a x}\hbox{\rm  d} x+ \int_m^\infty \frac{x}{A+x}e^{-\a x}\hbox{\rm  d} x
\cr & &  +  \frac{1}{\a}\big(\log (A+m)\big)e^{-\a m} + m(\log A+m) e^{-\a m}.
\end{eqnarray*}
Therefore,\begin{eqnarray*}
\a \int_m^\infty x \big(\log (A+x)\big) e^{-\a x}\hbox{\rm  d}x&\le & \frac{1}{\a^2(A+m)}e^{-\a m}+  \frac{1}{\a}e^{-\a m}  +  \frac{1}{\a}\big(\log (A+m)\big)e^{-\a m}
\cr & &  + m(\log A+m) e^{-\a m}.\end{eqnarray*}



\noi (iii) \rm We deduce from (i) that
 \begin{eqnarray*}
  \int_1^N \big(\log (A+x)\big) e^{-\a x}\hbox{\rm  d} x&=&\frac{1}{\a}\int_1^N\frac{1}{(A+x)}e^{-\a x}\hbox{\rm  d} x\cr & & \qq -\frac{1}{\a}\Big(\big(\log (A+1)\big) e^{-\a}-\log (A+N)\big) e^{-\a N}\Big).\end{eqnarray*}
 As
$\frac{1}{\a}\int_1^N\frac{1}{A+x}e^{-\a x}\hbox{\rm  d} x\le
\frac{1}{\a^2(A+1)}e^{-\a}$, letting $N$ tend to infinity gives,
 \begin{eqnarray*}
  \int_1^\infty \big(\log (A+x)\big) e^{-\a x}\hbox{\rm  d} x&\le &\frac{\log (A+1)}{\a} e^{-\a}+\frac{1}{\a^2(A+1)}e^{-\a}.\end{eqnarray*}
 \end{proof}

\begin{lemma}\label{E1}    Assume that $A\ge 1$, and $\a \ge 1$. Then,
 \begin{eqnarray*}
 & & \sum_{\m=0}^{\infty}\a \m \big(\log (A+\m)\big)e^{-\a \m}\ \le \  \a  \big(\log (A+1)\big)e^{-\a}+ 2\a  \big(\log (A+2)\big)e^{-2\a}
  \cr & &\qq\qq   + \Big\{3\a \log (A+3)+3 \log (A+3)+  \frac{1}{\a} \log (A+3) +  \frac{1}{\a}
+ \frac{1}{\a^2(A+3)}\Big\} e^{-3\a}.\end{eqnarray*}
\end{lemma}
\begin{proof}[\cmit Proof] \rm
As \begin{eqnarray*}\sum_{\m=0}^{\infty}\a \m \big(\log (A+\m)\big)e^{-\a \m}&=&\a  \big(\log (A+1)\big)e^{-\a}+ 2\a \big(\log (A+2)\big)e^{-2\a} \cr & & + 3\a \big(\log (A+3)\big)e^{-3\a}+ \a\sum_{\m=4}^\infty \m \big(\log (A+\m)\big)e^{-\a \m},
\end{eqnarray*}
by applying Lemma  \ref{phivar}-(ii), we get
   \begin{eqnarray*}
  \a\sum_{\m=4}^\infty \m \big(\log (A+\m)\big)e^{-\a \m}&\le &\a\int_{3}^\infty x \big(\log (A+x)\big)e^{-\a x} \hbox{\rm  d} x
  \cr &\le &  \frac{1}{\a^2(A+3)}e^{-3\a}+  \frac{1}{\a}e^{-3\a }  +  \frac{\log (A+3)}{\a}e^{-3\a}
+ 3(\log A+3) e^{-3\a}.\end{eqnarray*}
Whence,
 \begin{eqnarray*}
 & & \sum_{\m=0}^{\infty}\a \m \big(\log (A+\m)\big)e^{-\a \m}\ \le \  \a  \big(\log (A+1)\big)e^{-\a}+ 2\a  \big(\log (A+2)\big)e^{-2\a}
  \cr & & + \Big\{3\a \log (A+3)+3 \log (A+3)+  \frac{1}{\a} \log (A+3) +  \frac{1}{\a}
+ \frac{1}{\a^2(A+3)}\Big\} e^{-3\a}.\end{eqnarray*}
  \end{proof}
\begin{lemma}\label{E3a}    Under assumption  \eqref{min}
we have \begin{eqnarray*}\sum_{\m_s=0}^{\infty} \frac{\log
\big(\sum_{i=1}^{s} \m_i +h\big)}{p_s^{\m_s}}&\le &
\log
\big(\sum_{i=1}^{s-1} \m_i + h\big) + \frac{1}{p_{s}}\Big( 1 + \frac{1}{ \log p_{s}}\Big)\log
\big(\sum_{i=1}^{s-1} \m_i + h+1\big)
\cr & &\quad  +\frac{1}{(1+(\sum_{i=1}^{s-1} \m_i + 2))(\log p_s)^2p_s}.
\end{eqnarray*}
In particular,
\begin{eqnarray*}\sum_{\m_s=0}^{\infty} \frac{\log
\big(\sum_{i=1}^{s-1} \m_i +h\big)}{p_s^{\m_s}} &\le &
 \Big(1+ \frac{1}{p_{s}}\big( 1 + \frac{1}{ \log p_{s}} +\frac{1}{3(\log p_s)^2}\big) \Big)\log
\big(\sum_{i=1}^{s-1} \m_i + h+2\big).
\end{eqnarray*}\end{lemma}
 \begin{proof}[\cmit Proof]
  \rm  As
 \begin{eqnarray*}\sum_{\m=0}^{\infty} \big(\log (A+\m)\big) e^{-\a \m} &=& \log A +  \big(\log (A+1)\big) e^{-\a }+ \sum_{\m=2}^{\infty} \big(\log (A+\m)\big) e^{-\a \m}
 \cr &\le & \log A +  \big(\log (A+1)\big) e^{-\a }+ \int_{1}^{\infty} \big(\log (A+x)\big) e^{-\a x}\hbox{\rm  d} x\end{eqnarray*}
we deduce from Lemma \ref{phivar}-(iii),
\begin{equation}\label{sumphi2}\sum_{\m=0}^{\infty} \big(\log (A+\m)\big) e^{-\a \m}  \ \le \  \log A + e^{-\a}\Big(\log (A+1)+ \frac{\log (A+1)}{\a} + \frac{1}{\a^2(A+1)}\Big) .
\end{equation}

Consequently,
 \begin{eqnarray*}\sum_{\m_s=0}^{\infty} \frac{\log
\big(\sum_{i=1}^{s} \m_i + h\big)}{p_s^{\m_s}}&\le &
\log
\big(\sum_{i=1}^{s-1} \m_i + h\big)+ \frac{1}{p_{s}}\Big( 1 + \frac{1}{ \log p_{s}}\Big)\log
\big(\sum_{i=1}^{s-1} \m_i + h+1\big)\cr & &\quad  +\frac{1}{(1+(\sum_{i=1}^{s-1} \m_i + 2))(\log p_s)^2p_s}.
\end{eqnarray*}
Finally,
\begin{eqnarray*}\sum_{\m_s=0}^{\infty} \frac{\log
\big(\sum_{i=1}^{s-1} \m_i + h\big)}{p_s^{\m_s}}
&\le &
 \Big(1+ \frac{1}{p_{s}}\big( 1 + \frac{1}{ \log p_{s}}\big) +\frac{1}{3(\log p_s)^2}\Big)\log
\big(\sum_{i=1}^{s-1} \m_i + h+2\big)  .
\end{eqnarray*}
 \end{proof}
 \begin{corollary}\label{E2}     Assume that condition \eqref{min} is satisfied.
 \vskip 3 pt {\rm (i)} If  $\sum_{i=1}^{r-1}\m_i\ge 1$, then
\begin{align*}
 \sum_{\m_r=0}^{\a_r }&\frac{\m_r\log p_{r}}{p_r^{\m_r}}\log
\big(\sum_{i=1}^{r-1} \m_i + \m_r\big) \le   \frac{\log p_r}{p_r}\log \big( \sum_{i=1}^{r-1} \m_i + 1\big)+ \frac{2\log p_r}{p_r^2} \log \big( \sum_{i=1}^{r-1} \m_i + 2\big)  \cr &+ \frac{1}{p_r^3} \Big(3\log p_r + 3+ \frac{1}{\log p_r} \Big)\log \big( \sum_{i=1}^{r-1} \m_i + 3\big)  +  \frac{1}{p_r^3\log p_r}\Big(
1+  \frac{1}{(\sum_{i=1}^{r-1} \m_i + 3)\log p_r}\Big).
\end{align*}
Further,
\begin{eqnarray*}
\sum_{\m_r=0}^{\a_r }\frac{\m_r\log p_{r}}{p_r^{\m_r}}\log
\big(\sum_{i=1}^{r-1} \m_i + \m_r\big)& \le & 5\  \frac{\log p_r}{p_r}\log \big( \sum_{i=1}^{r-1} \m_i + 3\big).\end{eqnarray*}
 \vskip 3 pt {\rm (ii)} If  $\sum_{i=1}^{r-1}\m_i=0$, then
\begin{eqnarray*} \sum_{\m_r=0}^{\a_r }\frac{\m_r\log p_{r}}{p_r^{\m_r}}\log
\big(\sum_{i=1}^{r-1} \m_i + \m_r\big) &\le &  18\ \frac{\log p_{r}}{p_{r}} .\end{eqnarray*}\end{corollary}

\begin{proof}[\cmit Proof.] \rm (i) The first inequality follows from Lemma \ref{E1} with the choice $\a =\log p_r$ and $A=\sum_{i=1}^{r-1}\m_i$, noting that by assumption \eqref{min},
 $\a >1$. As $p_r\ge 3$, it is also  immediate that
\begin{align*}
 & \sum_{\m_r=0}^{\a_r }\frac{\m_r\log p_{r}}{p_r^{\m_r}} \log
\big(\sum_{i=1}^{r-1} \m_i + \m_r\big)
\cr &\  \le \Big\{3 \frac{\log p_r}{p_r}    + \frac{\log p_r }{9p_r} \Big(3 + \frac{3}{\log p_r}+ \frac{1}{(\log p_r)^2} \Big)\Big\}\log \big( \sum_{i=1}^{r-1} \m_i + 3\big)     +  \frac{1}{9p_r\log p_r}\Big(
1+  \frac{1}{4\log p_r}\Big)
\cr &\ \le \  5\ \frac{\log p_r}{p_r}\log \big( \sum_{i=1}^{r-1} \m_i + 3\big).\end{align*}

\noi (ii) If $\sum_{i=1}^{r-1}\m_i=0$, the sums relative  to  $\m_i$, $1\le i\le r-1$, do not contribute.
Further,
\begin{eqnarray*}
 \sum_{\m_r=0}^{\a_r }\frac{\m_r\log p_{r}}{p_r^{\m_r}}\log
\big(\sum_{i=1}^{r-1} \m_i + \m_r\big)&=&\sum_{\m_r=2}^{\a_r}\frac{\m_r\log p_{r}}{p_{r}^{\m_r}}\log \m_r\ =\  \sum_{\m=1}^{\a_r-1}\frac{(\m+1)\log p_{r}}{p_{r}^{\m+1}}\log (\m+1)
\cr &\le &\frac{1}{p_{r}}\Big\{\sum_{\m=1}^{\infty}\frac{\m\log p_{r}}{p_{r}^{\m}}\log (\m+1)+\sum_{\m=1}^{\infty}\frac{\log p_{r}}{p_{r}^{\m}}\log (\m+1)\Big\}   .
\end{eqnarray*}
Lemma \ref{E1} applied with $A=1$ and $\a =\log p_{r}$ gives the bound
 \begin{eqnarray*}
   \sum_{\m=1}^{\infty}\frac{\m\log p_{r}}{p_{r}^{\m}}\log (\m+1)
 &\le &
  \frac{(\log 2)\log p_{r}}{p_{r}}  + \frac{2(\log 3)\log p_{r}}{p_{r}^2}  + \frac{1}{p_{r}^3}\Big\{(6\log 2)(\log p_{r}) \cr & &  +6 \log 2+  \frac{2\log 2}{(\log p_{r})} +  \frac{1}{(\log p_{r})}
+ \frac{1}{4(\log p_{r})^2}\Big\}
\cr &\le &  8\Big(\frac{\log p_{r}}{p_{r}}+\frac{1}{p_{r}^3} \Big) .\end{eqnarray*}
Next estimate \eqref{sumphi2} applied with $A=1$ and $\a =\log p_{r}$,  further gives,
\begin{eqnarray*}\sum_{\m=1}^{\infty}\frac{\log p_{r}}{p_{r}^{\m}}\log (\m+1)  &\le &   \frac{1}{p_{r}}\Big(\log 2+ \frac{\log 2}{\log p_{r}} + \frac{1}{2(\log p_{r})^2}\Big) \ \le \ \frac{2}{p_{r}}.\end{eqnarray*}
Whence,
$ \sum_{\m_r=0}^{\a_r }\frac{\m_r\log p_{r}}{p_r^{\m_r}}\log
\big(\sum_{i=1}^{r-1} \m_i + \m_r\big) \le   18\ \frac{\log p_{r}}{p_{r}} .$\end{proof}

 \begin{remark}\rm

As $\log
\big(\sum_{i=1}^{s} \m_i + h\big)\le \log
\big(\O(n)+3\big)$,
one can deduce from Corollary  \ref{E2}-(ii) that
\begin{eqnarray*}
\Phi_2(r,n)
&\le &18\ \frac{\log p_{r}}{p_{r}}\log (\O(n)+3) \prod_{i=1}^{r}\Big(\frac{1}{1-p_i^{-1}}\Big)
.\end{eqnarray*}
So that by the observation made at the beginning of section \ref{s4},
\begin{eqnarray*}
\Phi_2(n)&\le &18\ \big(\log (\O(n)+3)\big)\Big( \sum_{j=1}^r\frac{\log p_{j}}{p_{j}}\Big) \prod_{i=1}^{r}\Big(\frac{1}{1-p_i^{-1}}\Big)
.\end{eqnarray*}
By combining this with the bound for  $\Phi_1(n)$ established in
Lemma \ref{phi1maj}, next using  inequality \eqref{convexdec}, gives
  \begin{align}\label{convexdec1} \Psi(n)\ \le\   \Big(\prod_{j=1 }^r \frac{1}{1-p_j^{ -1}} \Big)\bigg\{\sum_{i=1}^r \frac{(\log p_i)\big(\log\log p_i\big)}{p_i-1}  + 18\ \Big( \sum_{i=1}^r\frac{\log p_{i}}{p_{i}}\Big)\log (\O(n)+3) )\bigg\},\end{align}
recalling   that $r=\o(n)$. Whence by invoking Proposition \ref{tEZm}, noticing that $\o(n)\le\O(n)\le \log_{2} n$,
\begin{eqnarray*} \Psi(n)&\le & e^\g(1+o(1)) (\log\log n)^2\big(\log\log\log n+ 18 w( n)\big).\end{eqnarray*}
\vskip 3 pt  The finer estimate of  $\Psi(n)$ will be derived from a more precise study of the coefficients of $\Psi(r,n)$. This is the object of the next sub-section.
\end{remark}
 \subsection{Estimates of $\boldsymbol{ \Phi_2(r,n)}$.}
   \rm  ${}$
We  define successively
 \begin{eqnarray}\label{n1}{}\begin{cases} \ \ \ \  \  \ \  \  \m\, =\, (\m_1, \ldots, \m_r), \qq (\m_1, \ldots, \m_r)\in \displaystyle{\prod_{i=1}^r\big([0,\a_i]\cap \N\big)},
\cr  &\cr  \ \  \   p_\m(s)\, =\, p_1^{-\m_1}\ldots p_s^{-\m_s}, \qquad 1\le s\le r,
\cr &\cr 
\quad  \ \  \ \Pi_s\, =\, \displaystyle{\sum_{\m_1=0}^{\a_1}\ldots \sum_{\m_s=0}^{\a_s}p_\m(s)\, =\,  \prod_{\ell = 1}^s\Big(\frac{1-p_\ell^{-\a_\ell -1}}{1-p_\ell^{s -1}}\Big) }.
\end{cases}
\end{eqnarray}
 Next,
\begin{eqnarray*}  \Phi_s(h)= \sum_{\m_1=0}^{\a_1}\ldots \sum_{\m_s=0}^{\a_s}p_\m(s) \log
\big(\sum_{i=1}^{s} \m_i + h\big), \qq \quad 1\le s\le r-1.
\end{eqnarray*}

We also set

\begin{eqnarray}\label{n5} \begin{cases}c_1\, =\, 1, \qq c_2\, =\, \frac{2}{p_r},
 \qq c_3\, =\, \frac{1}{p_r^2}\big(3 + \frac{3}{\log p_r}+ \frac{1}{(\log p_r)^2}\big), 
 \cr   c_4\, =\, \frac{1}{p_r^3\log p_r}\big(1 + \frac{1}{3\log p_r}\big)
\cr   
c_0\, =\,  \frac{\log p_r}{ p_r}, \qq\qq\qq \   \, c\, =\, \sum_{i=1}^3 c_i,
\cr b_s\, =\, \frac{1}{ p_s}\big( 1+ \frac{1}{ \log p_s}\big), \qq  \ \  \b_s= \frac{1}{2 p_s (\log p_s)^2}. \qq\qq\qq\end{cases}
 \end{eqnarray}

   \subsubsection{\cmit Recurrence inequality.}\label{subsection4.2.1}
We deduce from the first part of Lemma \ref{E3a} that
\begin{eqnarray*}\Phi_{s}(h)&=& \sum_{\m_1=0}^{\a_1}\ldots \sum_{\m_{s-1}=0}^{\a_{s-1}}p_\m(s-1)\bigg\{\sum_{\m_s=0}^{\a_s}p_s^{-\m_s} \log
\big(\sum_{i=1}^{s} \m_i + h\big)\bigg\}
\cr &\le & \sum_{\m_1=0}^{\a_1}\ldots \sum_{\m_{s-1}=0}^{\a_{s-1}}p_\m(s-1)\bigg\{\sum_{\m_s=0}^{\infty}p_s^{-\m_s} \log
\big(\sum_{i=1}^{s} \m_i + h\big)\bigg\}
\cr &\le &\sum_{\m_1=0}^{\a_1}\ldots \sum_{\m_{s-1}=0}^{\a_{s-1}}p_\m(s-1)\bigg\{\log
\big(\sum_{i=1}^{s-1} \m_i + h\big) +
\cr & & \quad \frac{1}{p_{s}}\Big( 1 + \frac{1}{ \log p_{s}}\Big)\log
\big(\sum_{i=1}^{s-1} \m_i + h+1\big)+\frac{1}{(1+(\sum_{i=1}^{s-1} \m_i + 2))(\log p_s)^2p_s}\bigg\}
\cr &\le &\Phi_{s-1}(h)+  \frac{1}{p_{s}}\Big( 1 + \frac{1}{ \log p_{s}}\Big)\Phi_{s-1}(h+1)+ \frac{1}{2(\log p_s)^2p_s}\Pi_{s-1}.
\end{eqnarray*}

\vskip 5 pt

Whence with the previous notation,
\begin{lemma}\label{E3} Under assumption \eqref{min},
we have for $s=2,\ldots ,
r-1$,\begin{eqnarray*}\Phi_{s}(h)&\le &\Phi_{s-1}(h)+  b_s\Phi_{s-1}(h+1)+\b_s\Pi_{s-1}.
\end{eqnarray*}
\end{lemma}


\vskip 5 pt
  Now by using estimate (i) of Corollary  \ref{E2}  and the notation introduced, we have, under assumption \eqref{min}, if
$\sum_{i=1}^{r-1}\m_i\ge 1$,
\begin{align}\label{h1234}
 \sum_{\m_r=0}^{\a_r }&\frac{\m_r\log p_{r}}{p_r^{\m_r}}\log
\big(\sum_{i=1}^{r-1} \m_i + \m_r\big) \le   \frac{\log p_r}{p_r}\log \big( \sum_{i=1}^{r-1} \m_i + 1\big)+ \frac{2\log p_r}{p_r^2} \log \big( \sum_{i=1}^{r-1} \m_i + 2\big)  \cr &+ \frac{1}{p_r^3} \Big(3\log p_r + 3+ \frac{1}{\log p_r} \Big)\log \big( \sum_{i=1}^{r-1} \m_i + 3\big)  +  \frac{1}{p_r^3\log p_r}\Big(
1+  \frac{1}{(\sum_{i=1}^{r-1} \m_i + 3)\log p_r}\Big)
\cr & \le  c_0c_1
 \log \big( \sum_{i=1}^{r-1} \m_i + 1\big)+ c_0c_2
 \log \big( \sum_{i=1}^{r-1} \m_i + 2\big)
 +c_0c_3
 \log \big( \sum_{i=1}^{r-1} \m_i + 3\big)  +  c_4
\cr & =c_0\sum_{h=1}^3c_i \log \big( \sum_{i=1}^{r-1} \m_i + h\big)+c_4,\end{align}
since $\frac{1}{p_r^3\log p_r}\big(1+  \frac{1}{(\sum_{i=1}^{r-1} \m_i + 3)\log p_r}\big)\le c_4$.
\vskip 3 pt
Therefore, under assumption \eqref{min}, if
$\sum_{i=1}^{r-1}\m_i\ge 1$,
\begin{eqnarray}\Phi_2(r,n)&\le &   c_0 \underbrace{\sum_{h=1}^3 c_h \Phi_{r-1}(h)}_{(1)} + c_4 \Pi_{r-1}.\end{eqnarray}
Indeed,
\begin{eqnarray*}\Phi_2(r,n)&=& \sum_{\m_1=0}^{\a_1}\ldots \sum_{\m_{r-1}=0}^{\a_{r-1}} \frac{1}{ p_1^{\m_1}\ldots p_{r-1}^{\m_{r-1}}}\sum_{\m_r=0}^{\a_r}\frac{\m_r \log p_r}{p_r^{\m_r}}\log \Big[\sum_{i=1}^{r-1} \m_i + \m_r\Big]
\cr &\le & \sum_{\m_1=0}^{\a_1}\ldots \sum_{\m_{r-1}=0}^{\a_{r-1}} \frac{1}{ p_1^{\m_1}\ldots p_{r-1}^{\m_{r-1}}}\Big\{c_0\sum_{h=1}^3c_i \log \big( \sum_{i=1}^{r-1} \m_i + h\big)+c_4\Big\}
\cr &= &  c_0 \underbrace{\sum_{h=1}^3 c_h \Phi_{r-1}(h)}_{(1)} + c_4 \Pi_{r-1}.\end{eqnarray*}
By applying  the recurrence inequality with $s=r-1$ to $\Phi_{r-1}(h)$, one gets
 \begin{eqnarray*}
\Phi_2(r,n) &\le &  c_0 \underbrace{\sum_{h=1}^3 c_h \big[\Phi_{r-2}(h)}_{(1)}+ \underbrace{b_{r-1}\Phi_{r-2}(h+1)}_{(2)} \big]+c_0c\b_{r-1}\Pi_{r-2}+ c_4 \Pi_{r-1}.\end{eqnarray*}
By applying this time  the recurrence inequality to   $\Phi_{r-2}(h)$, one also gets
 \begin{eqnarray*}
 \Phi_2(r,n) &\le &  c_0 \underbrace{\sum_{h=1}^3 c_h \Phi_{r-3}(h)}_{(1)}+c_0\underbrace{\sum_{h=1}^3 c_h b_{r-2}\Phi_{r-3}(h+1)}_{(3)}+c_0c b_{r-2}\Pi_{r-3}
 \cr& &+c_0\underbrace{\sum_{h=1}^3 c_h b_{r-1}\Phi_{r-3}(h+1)}_{(2)}+c_0\underbrace{\sum_{h=1}^3 c_h b_{r-1}b_{r-2}\Phi_{r-3}(h+2)}_{(4)}+c_0cb_{r-1}\b_{r-2}\Pi_{r-3}
 \cr& &+c_0c\b_{r-1}\Pi_{r-2}+ c_4 \Pi_{r-1}.\end{eqnarray*}
One easily verifies (see  expressions underlined  by (1)) that the  coefficient of $\Phi_{r-1}(h)$ is the same as the one of  $\Phi_{r-2}(h)$  and $\Phi_{r-3}(h)$. So is also the case for  $\Phi_{r-2}(h+1)$, see  expressions underlined  by (2). New expressions underlined by (3),\,(4) and linked to $\Phi_{r-3}(h+1)$, $\Phi_{r-3}(h+2)$ appear.
\vskip 2 pt
Each new coefficient is kept until the end of the iteration process generated by the recurrence inequality of Lemma \ref{E3}.
\vskip 2 pt
We also verify, when applying this inequality, that we pass from a majoration expressed by  $\Phi_{r-1}(h)$, $\Pi_{r-1}$, {\cmit
 uniquely}, to a majoration  expressed by $\Phi_{r-2}$ (in $h$ or $h+1$) and $\Pi_{r-2}$, $\Pi_{r-1}$ {\cmit uniquely}.
\vskip 2 pt
This rule is general, and one verifies that when iterating this recurrence relation, we  obtain at each step a bound depending on $\Phi_{r-d}$ and the products $\Pi_{r-d}, \Pi_{r-d+1},\ldots,$ $ \Pi_{r-1}$ only.
\vskip 4 pt
$\underline{\hbox{\cmssqi Binary tree}}$\,: The shift of length $h$ or $h+1$ generates a binary tree whose branches are at each division (steps corresponding to the preceding iterations), either stationary\,:  $\Phi_{r-d}(h)\to \Phi_{r-d-1}(h)$, or creating new coefficients\,: $\Phi_{r-d}(h)\to \Phi_{r-d-1}(h+1)$.
One can represent this by the diagram below drawn from Lemma \ref{E3}.  \vskip 5 pt
\centerline{ ${}_\downarrow$ \hbox{\small shift\,+1, new coefficients\ }${}_\downarrow$\hskip +52 pt}
\vskip -15 pt \begin{eqnarray*}\Phi_{s}(h)&\le &\Phi_{s-1}(h)+  b_s\Phi_{s-1}(h+1)+\b_s\Pi_{s-1}.
\end{eqnarray*}
\centerline{ ${}^\uparrow$ \hbox{\small stationarity} ${}^\uparrow$ \hskip +120 pt}
\vskip -3pt \centerline{\sevenrm Figure 1.}\par

\par
\vskip 6 pt

\noi

 Before continuing, we recall that by \eqref{sumphi2},
\begin{eqnarray*}\sum_{\m=0}^{\a_s} \big(\log (A+\m)\big) e^{-\a \m}  &\le & \log A + e^{-\a}\Big(\log (A+1)+ \frac{\log (A+1)}{\a} + \frac{1}{\a^2(A+1)}\Big) .
\end{eqnarray*}

\noi Thus
\begin{eqnarray*} \Phi_1(v)&\le & \sum_{\m_1=0}^\infty p_\m(1)\log \Big(\sum_{i=1}^v\m_i+1\Big)
\ =\  \sum_{\m_1=0}^\infty \frac{\log (v+\m)}{p_1^\m}
\cr &\le & \log v + \frac{1}{p_1}\Big( \log (v+1) +\frac{\log (v+1)}{\log p_1}+ \frac{1}{v(\log p_1)^2}\Big) \qq (v\ge 1).
\end{eqnarray*}

Hence,
\begin{eqnarray*} \Phi_1(h) &\le & C\log h.
\end{eqnarray*}
One easily verifies that the   $d$-tuples formed with the $b_i$ have all  $\Phi_{r-x}(h+d)$ as factor. The terms having   $\Phi_{r-\cdot}(h+\cdot)$ as factor are forming the sum
\begin{eqnarray} \label{somme} c_0\, \sum_{d=1}^{r-1}\Big( \sum_{1\le i_1<\ldots <i_d<r} b_{i_1}\ldots b_{i_d}\Big) \Phi_1(h+d),
\end{eqnarray}
once the iteration process achieved, that is after having applied $(r-1)$ times the recurrence inequality of Lemma  \ref{E3}.

This sum can thus  be bounded from above by (recalling that $h=1,2$ or $3$)
 \begin{eqnarray*} c_0\sum_{d=1}^{r-1}(\log d)\, \Big( \sum_{1\le i_1<\ldots <i_d<r} b_{i_1}\ldots b_{i_d}\Big)
 . \end{eqnarray*}
But, for all positive integers $a_{ 1},\ldots, a_{r}$ and $1\le d\le r$, we have,
$$\Big(\sum_{i=1}^r a_i\Big)^d\ge d!
\sum_{ 1\le i_1<\ldots<i_d\le r }  a_{i_1}\ldots a_{i_d}.
$$
Thus
 \begin{eqnarray*} \sum_{d=1}^{r-1}(\log d)\, \Big( \sum_{1\le i_1<\ldots <i_d<r} b_{i_1}\ldots b_{i_d}\Big)
&\le &  \sum_{d=1}^{r-1}\frac{(\log d)}{d!}\, \Big( \sum_{i=1}^r b_{i}\Big)^d
  . \end{eqnarray*}
 As moreover,
 $$b_{i}=\frac{1}{p_{i}}\big(1 + \frac{1}{\log p_{i+1}}\big)\le \frac{1}{p(i)} + \frac{1}{p(i)\log p(i)},$$
 one has by means of  \eqref{p(i)est},
 \begin{eqnarray*}  \sum_{i=1}^r b_{i}&\le & \sum_{i=1}^r\big(\frac{1}{i \log i} + \frac{1}{i (\log i)^2}\big)
\ \le \  \log\log r +C. \end{eqnarray*}
Thus  \begin{eqnarray*} \sum_{d=1}^{r-1}(\log d)\, \Big( \sum_{1\le i_1<\ldots <i_d<r} b_{i_1}\ldots b_{i_d}\Big)
&\le & C \sum_{d=1}^{r-1}\frac{(\log d)}{d!}\, (\log\log r +C)^d
 . \end{eqnarray*}
 On the one hand,
 \begin{eqnarray*} \sum_{\log d\le  1+\e +\log\log\log r}\frac{(\log d)}{d!}\,(\log\log r +C)^d&\le & \big(1+\e +\log\log\log r\big) \sum_{d>1}\frac{(\log\log r +C)^{d}}{d!}
 \cr &\le &C \big(1+\e +\log\log\log r\big) \log r.
 \end{eqnarray*}
On the other,  utilizing the classical estimate  $d\,!\ge C \sqrt d \,d^d\,e^{-d}$, one has
 \begin{eqnarray*} \sum_{\log d> 1+\e+\log\log\log r}\frac{(\log d)}{d!}\, (\log\log r)^d&\le &\sum_{\log d> 1+\e+\log\log\log r}\frac{(\log d)}{\sqrt d}\,e^{-d(\log d-1 - \log\log\log r)}
 \cr &\le & \sum_{d>1}\frac{(\log d)}{\sqrt d}\,\,e^{-\e d}<\infty.
 \end{eqnarray*}
 One thus deduces, concerning the sum in  \eqref{somme} that,
 \begin{equation} \label{somme1} c_0\sum_{d=1}^{r-1}\Big( \sum_{1\le i_1<\ldots <i_d<r} b_{i_1}\ldots b_{i_d}\Big) \Phi_1(h+d)\ \le \ C\frac{\log p_r}{p_r}\big(1+\log\log\log r\big)\log r;
\end{equation}
  \vskip 4 pt
\subsection{\cmit Coefficients related to  $   \Pi_s$.}
\vskip 3 pt
By applying the recurrence inequality  (Lemma \ref{E3}), one successively generates
 \begin{eqnarray*} & & c_4\Pi_{r-1} 
  \cr & & c_4\Pi_{r-1} + c_0c\b_{r-1}\Pi_{r-2} 
\cr & & c_4\Pi_{r-1} + c_0c\b_{r-1}\Pi_{r-2} + c_0c\b_{r-2}\big(1 + b_{r-1}b_{r-2} \big)\Pi_{r-3} 
 \cr  & &c_4\Pi_{r-1} + c_0c\b_{r-1}\Pi_{r-2} + c_0c\b_{r-2}\big(1 + b_{r-1}b_{r-2} \big)\Pi_{r-3}
 \cr & &\qq\quad \       +c_0c\b_{r-3}\big( 1+ b_{r-2} + b_{r-1}  + b_{r-1} b_{r-2}\big)\Pi_{r-4}.
  \end{eqnarray*}
 $\underline{\hbox{\cmssqi Coefficients}}$\,:
 \begin{eqnarray*} \qq\qq \Pi_{r-1}: c_4\qq\qq \qq\qq  \, \Pi_{r-2}: c_0c\b_{r-1} \qq\qq\qq\qq\qq\qq\qq\qq\cr
\Pi_{r-3}: c_0c\b_{r-2}(1+ b_{r-1})\qq  \Pi_{r-4}: c_0c\b_{r-3}(1+ b_{r-2}+b_{r-1}+b_{r-1}b_{r-2}).\qq\qq \end{eqnarray*}
 It is easy to check that the coefficients $\Pi_{r-x}$ are exactly those of  $\Phi_{r-x+1}(.)$  affected with the factor $c_0c\b_{r-x+1}$.
The products form the sum
\begin{eqnarray} \label{sumpi} c_0c \sum_{d=0}^{r-2}\b_{r-d}\Big(1+ \sum_{1\le i_1<\ldots <i_d<r} b_{r-i_1}\ldots b_{r-i_d}\Big)\Pi_{r-d-1}.
\end{eqnarray}
By \eqref{p(i)est}, one has
\begin{eqnarray}\label{beta}\b_j= \frac{1}{2 p_j (\log p_j)^2}\le \frac{1}{2 p(j) (\log p(j))^2}\le \frac{1}{2j (\log j)^3},\qq \hbox{if $j\ge 2$,}
\end{eqnarray}
Moreover, \eqref{p(i)est} and \eqref{prod} imply that
\begin{eqnarray*}
 \Pi_j= \prod_{\ell =1}^j \Big(\frac{1}{1-\frac{1}{p_\ell}} \Big)&\le &\prod_{\ell =1}^j \Big(\frac{1}{1-\frac{1}{p(\ell)}} \Big)\ \le \  \prod _{p\le j(\log j +\log\log j)}\Big(\frac{1}{1- \frac{1}{p}}\Big)
\cr &\le &  C (\log j) \, .
\end{eqnarray*}
We now note that by definition of  $\Pi_j$, we also have
$$\Pi_j\le \max_{\ell \le 5}\prod_{p\le p(\ell)} \frac{1}{1 -\frac{1}{p}}= C_0.$$
We deduce that
\begin{eqnarray}\label{Piest}\Pi_j &\le &C( \log j)  ,\qq\qq  \hbox{if $j\ge 2$.}
\end{eqnarray}
Consequently,   \eqref{Piest} and \eqref{beta} imply that
\begin{eqnarray}\label{betapi}\b_{j+1}\Pi_j &\le & \frac{C}{j (\log j)^2} ,\qq\qq  \hbox{if $j\ge 2$.}
\end{eqnarray}

This implies that  the sum in  \eqref{sumpi} can be bounded as follows:
\begin{eqnarray} \label{estsumpi}
& & c_0c \sum_{d=0}^{r-2}\b_{r-d}\Big(1+ \sum_{1\le i_1<\ldots <i_d<r} b_{r-i_1}\ldots b_{r-i_d}\Big)\Pi_{r-d-1}
\cr &\le &c_0c  \prod _{i=1}^{r-2}\big( 1+b_{r-i}\big)\cdot  \sum_{d=0}^{r-2}\b_{r-d}\, \Pi_{r-d-1}
\ =\ c_0c  \prod _{j=2}^{r-1}\big( 1+b_{j}\big)\cdot  \sum_{d=0}^{r-2}\b_{r-d}\, \Pi_{r-d-1}
\cr &\le  &c_0c\, C \prod _{j=2}^{r-1}\big( 1+b_{j}\big)\cdot  \sum_{d=0}^{r-2}\frac{1}{(r-d)\big(\log (r-d)\big)^2}
 \cr &\le  &c_0c\, C \prod _{j=2}^{r-1}\big( 1+b_{j}\big)\cdot  \sum_{\d=2}^{\infty}\frac{1}{\d (\log \d)^2}
 \cr &\le  & c_0c\, C   \prod _{j=2}^{r-1}\big( 1+b_{j}\big). \end{eqnarray}
We recall that  \begin{eqnarray*} \sum_{p\le x }\frac{1}{p}\le  \log\log x +C.
 \end{eqnarray*}
 See for instance \cite{RS}, inequality (3.20). Thus,
 \begin{eqnarray}\label{1plusbi}  \prod_{i=1}^{r}\big(1+b_i\big)&\le & C\, \log r.\end{eqnarray}
Now estimate \eqref{1plusbi} implies that
\begin{eqnarray}\label{estsumpi2}
 c_0c \sum_{d=0}^{r-2}\b_{r-d}\Big(1+ \sum_{1\le i_1<\ldots <i_d<r} b_{r-i_1}\ldots b_{r-i_d}\Big)\Pi_{r-d-1}&\le & c_0c\, C  \log r
\cr  &\le &  C\,  \frac{ \log  p_r}{p_r}  \log r.\end{eqnarray}
We thus deduce from  \eqref{somme1} and \eqref{sumpi} that
 \begin{eqnarray}\label{estS2}
\Phi_2(r,n)&\le   & C\,\frac{\log p_r}{p_r}\big(1+\log\log\log r\big)\log r + C\,  \frac{ \log  p_r}{p_r}  \log r
\cr &\le &     C\,\frac{\log p_r}{p_r}(\log r)(\log\log\log r) \, .\end{eqnarray}

As a result, by taking account of the observation made at the beginning of section \ref{s4}, we obtain
\begin{equation}\label{estS2a}
\Phi_2(n) \le     C\,  (\log\log\log r)(\log r)\sum_{i=1}^r\frac{ \log  p_i}{p_i}\ =\ C\,  (\log\log\log r)(\log r)w(n)
\, .\end{equation}
\vskip 3 pt By combining \eqref{estS2} with the upper estimate  $\Phi_1(n)$ established at Lemma  \ref{phi1maj} and using inequality \eqref{convexdec}, we arrive at
 \begin{equation}\label{convexdec1a} \Psi(n)\le  \Big(\prod_{j =1 }^r \frac{1}{1-p_j^{ -1}} \Big)\sum_{i=1}^r \frac{(\log p_i)\big(\log\log p_i\big)}{p_i-1} + C\,  (\log\log\log r)(\log r)w(n)
,\end{equation}
recalling that $p_j\ge 3$ by assumption \eqref{min}.



 \section{\bf Proof of Theorem \ref{t1}.}\label{s5}

 First we prove  inequality \eqref{convexdec}.   We recall   the convention   $0\log 0=0$. Inequality \eqref{convexdec} is an immediate consequence of the  following convexity lemma.
\begin{lemma}\label{lconvexe} For any integers $\mu_i\ge 0$, $p_j\ge 2$, we have
\begin{eqnarray*} \sum_{i=1}^{r}\big(\m_i\log p_i\big) \log \Big(\sum_{i=1}^{r} \m_i \log p_i\Big)&\le & \sum_{i=1}^{r} \m_i\big(\log p_i\big)\big(\log\log p_i\big)
\cr & &\qq +\sum_{i=1}^{r} \m_i\big(\log p_i\big)\log \Big(\sum_{i=1}^{r} \m_i \Big) .
\end{eqnarray*}
 \end{lemma}

\begin{proof}
  We may restrict to the case
$\sum_{i=1}^{r} \m_i\ge 1$,
since otherwise
the  inequality is trivial.
Let  $M=\sum_{i=1}^{r} \m_i$ and write that
\begin{eqnarray*} \sum_{i=1}^{r}\m_i\big(\log p_i\big) \log \Big(\sum_{i=1}^{r} \m_i \log p_i\Big)
&= & M\bigg\{ \sum_{i=1}^{r}\frac{\m_i}{M}\big(\log p_i\big) \log  \Big\{\sum_{i=1}^{r} \frac{\m_i}{M} \log p_i\Big\} \cr & & \quad +\sum_{i=1}^{r}\frac{\m_i}{M}\big(\log p_i\big)(\log M) \bigg\} .
\end{eqnarray*}
By using convexity of $\psi(x)=x\log x$
 on $\R_+$, we get
 $$ \sum_{i=1}^{r}\frac{\m_i}{\sum_{i=1}^{r}\m_i}\big(\log p_i\big) \log  \Big\{\sum_{i=1}^{r} \frac{\m_i}{\sum_{i=1}^{r}\m_i} \log p_i\Big\}\le \sum_{i=1}^{r}\frac{\m_i}{\sum_{i=1}^{r}\m_i}\big(\log p_i\big)\big(\log \log p_i\big).$$
Thus \begin{eqnarray*} \sum_{i=1}^{r}\big(\m_i\log p_i\big) \log \Big(\sum_{i=1}^{r} \m_i \log p_i\Big) \le   \sum_{i=1}^{r} \m_i\big(\log p_i\big)\big(\log\log p_i\big)+\sum_{i=1}^{r} \m_i\big(\log p_i\big)\log \Big(\sum_{i=1}^{r} \m_i \Big) .
\end{eqnarray*}
\end{proof}

 The odd case (i.e. condition \eqref{min} is satisfied) is obtained by combining  \eqref{estS2} with Corollary \ref{ests1} and utilizing   inequality \eqref{convexdec}. Since $r\le \log n$,   by taking account of estimate of  $w(n)$ given in  \eqref{wdavenport}, we get
  \begin{eqnarray}\label{convexdec2} \Psi(n)&\le& e^\g(1+o(1)) (\log\log n)^2(\log \log\log n) + C\, (\log\log\log\log n)(\log\log n)^2
  \cr &= & e^\g(1+o(1)) (\log\log n)^2(\log \log\log n).
  \end{eqnarray}
\vskip 5 pt
To pass from the odd case to the general case is not easy. This step will  necessitate an extra analysis of some  other  properties of  $\Psi(n)$.
\vskip 3 pt
 We first exclude the trivial case when $n$ is a pure power of $2$, since $\Psi(2^k) \le C$ uniformly over $k$, and  $C$ is a finite constant.
\vskip 3 pt Now if $2$ divides $n$, writing $n=2^vm$, $2 \not| m$, we have
\begin{eqnarray*}
  \Psi(n)&=&\sum_{d|n} \frac{(\log d )(\log\log d)}{d}\ =\ \sum_{k=0}^v\sum_{\d| m}\frac{(\log (2^k\d) )(\log\log (2^k\d))}{2^k\d}.
.
\end{eqnarray*}
As the function $x\mapsto  \frac{(\log   x )(\log\log x)}{x}$ decreases on $[x_0,\infty)$ for some  positive real $x_0$, we can write
\begin{eqnarray*} & &
\sum_{k=0}^v\frac{(\log (2^k\d) )(\log\log (2^k\d))}{2^k\d}
\cr &\le &\sum_{k=0}^{k_0-1}\frac{(\log (2^k\d) )(\log\log (2^k\d))}{2^k\d} + \sum_{k=k_0+1}^v\frac{(\log (2^k\d) )(\log\log (2^k\d))}{2^k\d}
\cr &\le &\sum_{k=0}^{k_0-1}\frac{(\log (2^k\d) )(\log\log (2^k\d))}{2^k\d} + \int_{2^{k_0}\d}^{\infty} \frac{(\log u )(\log\log u)}{u^2}\dd u,
\end{eqnarray*}
where  $k_0$ is depending on $x_0$ only. Moreover
$$\Big(\frac{(\log u )(\log\log u)}{u}\Big)'\ge  - \frac{(\log u )(\log\log u)}{u^2}.$$
Thus
\begin{eqnarray*} & &\sum_{k=0}^v\frac{(\log (2^k\d) )(\log\log (2^k\d))}{2^k\d}
\cr &\le &\sum_{k=0}^{k_0-1}\frac{(\log (2^k\d) )(\log\log (2^k\d))}{2^k\d} + \frac{(\log (2^{k_0}\d) )(\log\log (2^{k_0}\d))}{2^{k_0}\d},
\end{eqnarray*}
whence
\begin{eqnarray}\label{psik_0}\Psi(n)
  &\le & \sum_{k=0}^{k_0}\sum_{\d|m} \frac{(\log (2^k\d) )(\log\log (2^k\d))}{2^k\d}.
\end{eqnarray}

 Let $m=p_1^{b_1}\ldots p_{\m}^{b_{\m}}$. We have by \eqref{convexdec1}
 \begin{eqnarray*} \Psi(m)&\le&  \Big(\prod_{j =1 }^{\m} \frac{1}{1-p_j^{ -1}} \Big)\sum_{i=1}^{\m} \frac{(\log p_i)\big(\log\log p_i\big)}{p_i-1} + C\,  (\log\log\log \m)(\log \m)w(m)
\cr &\le&  \Big(\prod_{j =2 }^\m \frac{1}{1-p(j)^{ -1}} \Big)\sum_{i=1}^\m \frac{(\log p(i))\big(\log\log p(i)\big)}{p(i)-1} + C\,  (\log\log\log \m)(\log \m)w(m)
\cr &=& \frac12\, \Big(\prod_{j =1 }^\m \frac{1}{1-p(j)^{ -1}} \Big)\sum_{i=1}^\m \frac{(\log p(i))\big(\log\log p(i)\big)}{p(i)-1} + C\,  (\log\log\log \m)(\log \m)w(m)
\cr &\le & \frac{ e^{\g}}2\,  \big( \log \m + \mathcal O(1) \big)\sum_{i=1}^\m \frac{(\log p(i))\big(\log\log p(i)\big)}{p(i)-1} + C\,  (\log\log\log \m)(\log \m)w(m),\end{eqnarray*}
by using Mertens' estimate \eqref{prod} and since $p(\m)\sim \m\log \m$.
Furthermore by   using estimate \eqref{phi1.sumr}, and since $2^\m\le m$ we get
 \begin{eqnarray}\label{psi(m)est} \Psi(m)
 &\le & \frac{ e^{\g}}2\,  \big( \log \m + \mathcal O(1) \big)(1+\e)(\log \m)(\log\log \m) + C\,  (\log\log\log \m)(\log \m)w(m)
\cr &\le & \frac{ e^{\g}}2\,  \big( \log \frac{\log m}{\log2} + \mathcal O(1) \big)(1+\e)(\log \frac{\log m}{\log2})(\log\log \frac{\log m}{\log2}) \cr & &
+ C\,  (\log\log\log \frac{\log m}{\log2})(\log \frac{\log m}{\log2})(1+o(1))\log\log m
\cr &\le & \frac{ e^{\g}}2\,(1+2\e)  (\log \log m)^2(\log\log\log m),
\end{eqnarray}
for $m$ large.

 Now let $\psi(2^km)=\sum_{\d|m} \frac{(\log (2^k\d) )(\log\log (2^k\d))}{\d}$,  $1\le k\le k_0$. If $n$ is not a pure power of $2$, then its odd component $m$ tends to infinity with $n$. Thus with \eqref{psik_0},
 \begin{eqnarray}\label{estevencase} \frac{\Psi(n)}{\big(\log \log n)^2(\log\log\log n)}
 &\le & \sum_{k=0}^{k_0} \frac{1}{2^k}\sum_{\d|m} \frac{\frac{(\log (2^k\d) )(\log\log (2^k\d))}{\d}}{ (\log \log m)^2(\log\log\log m)}
.
\end{eqnarray}
But
 \begin{eqnarray}\label{estevencasea}\frac{(\log (2^k\d) )(\log\log (2^k\d))}{\d}&=&\frac{(k(\log 2) )(\log\log (2^k\d))+ (\log \d)(\log\log (2^k\d) }{\d}
 \cr &\le &k_0(\log 2) \frac{\log\big(k_0(\log 2)+\log\d\big)}{\d}+ \frac{(\log \d)(\log\log (2^{k_0}\d) }{\d}.
\end{eqnarray}
Now we have the inequality: $\log \log (a+x)\le \log (b\log x)$   where $b\ge (a+e)$ and $a\ge 1$, which is  valid for $x\ge e$. Thus
\begin{eqnarray}\label{estevencaseb}\log\big(k_0(\log 2)+\log\d\big)\le \log (k_0\log 2+ e)+\log \log \d.
\end{eqnarray}
Consequently
   \begin{eqnarray}\label{estevencase1}
   & & \sum_{k=0}^{k_0} \frac{1}{2^k}\sum_{\d|m} \frac{k_0(\log 2) \frac{\log (k_0(\log 2)+\log\d )}{\d}}{(\log \log m)^2(\log\log\log m)}
\cr &\le &\sum_{k=0}^{k_0} \frac{1}{2^k}\sum_{\d|m} \frac{k_0(\log 2)\frac{\log (k_0\log 2+ e)}{\d}}{(\log \log m)^2(\log\log\log m)}
\cr & & \quad + \sum_{k=0}^{k_0} \frac{1}{2^k}\sum_{\d|m} \frac{k_0(\log 2) \frac{\log\log\d}{\d}}{(\log \log m)^2(\log\log\log m)}
\cr&\le &2k_0(\log 2)\big(\log (k_0\log 2+ e)\big) \frac{\s_{-1}(m)}{(\log \log m)^2(\log\log\log m)}
\cr & & \quad + \frac{2k_0(\log 2)}{(\log \log m)^2(\log\log\log m)}\sum_{\d|m}   \frac{\log\log\d}{\d}
\cr&\le &C(k_0)\Big\{\frac{1}{\log \log m(\log\log\log m)}
+ \frac{\s_{-1}(m)}{(\log \log m)(\log\log\log m)}\Big\}
\cr&\le &\frac{C(k_0)}{\log\log\log m} \quad \to \ 0\quad \hbox{ as $m$ tends to infinity}.
\end{eqnarray}

Further
 \begin{eqnarray} \label{estevencase2}\sum_{k=0}^{k_0} \frac{1}{2^k}\frac{\sum_{\d|m}\frac{(\log \d)(\log\log (2^{k_0}\d) }{\d}}{(\log \log m)^2(\log\log\log m)}
& \le & \sum_{k=0}^{k_0} \frac{1}{2^k}\sum_{\d|m}\frac{(\log \d)( \log (k_0\log 2+ e)+\log \log \d) }{\d(\log \log m)^2(\log\log\log m)}
\cr &\le &\frac{\log (k_0\log 2+ e)}{(\log \log m)^2(\log\log\log m)}\,\sum_{k=0}^{k_0} \frac{1}{2^k}\sum_{\d|m}\frac{(\log \d)  }{\d}
\cr & &\quad+2\,\frac{\Psi(m)}{(\log \log m)^2(\log\log\log m)}
\cr &\le &\frac{2\log (k_0\log 2+ e)\,\s_{-1}(m)}{(\log \log m)(\log\log\log m)}
\cr & &\quad+2\,\frac{\Psi(m)}{(\log \log m)^2(\log\log\log m)}
\cr &\le &\frac{C(k_0)}{\log\log\log m}
+2\,\frac{ e^{\g}}2 \,(1+2\e)\,\frac{ (\log \log m)^2(\log\log\log m)}{(\log \log m)^2(\log\log\log m)}
\cr &\le &\frac{C(k_0)}{\log\log\log m}
+ e^{\g} \,(1+2\e)\, ,
\end{eqnarray}
for $m$ large, where we used estimate \eqref{psi(m)est}.
\vskip 3pt
Plugging estimates \eqref{estevencase1} and \eqref{estevencase2} into \eqref{estevencase} finally leads, in view of \eqref{estevencasea}, to
\begin{eqnarray}\label{estevencasef} \frac{\Psi(n)}{(\log \log n)^2(\log\log\log n)}
  &\le& \frac{C}{\log\log\log m}
+ e^{\g} \,(1+2\e)\,
\end{eqnarray}
for $m$ large, where $C$ depends on $k_0$ only. As $\e$ can be arbitrary small, we finally obtain
\begin{eqnarray}\label{evencasef} \limsup_{n\to \infty}\frac{\Psi(n)}{(\log \log n)^2(\log\log\log n)}
  &\le&  e^{\g}.
\end{eqnarray}
This establishes Theorem \ref{t1}.


 \section{\bf Complementary results.}\label{s2}
 In this section we prove complementary estimates $\Phi_1$, $\Phi_2$ and $\Psi$, notably estimates \eqref{phipsi} and \eqref{Phi1est}
 \subsection{Upper estimates.}
\begin{lemma}\label{phi1maj} We have the following estimate,
\begin{eqnarray*} \Phi_1(n)&\le&\Big(\prod_{j =1 }^r \frac{1}{1-p_j^{ -1}} \Big)\sum_{i=1}^r \frac{(\log p_i)\big(\log\log p_i\big)}{p_i-1}.
\end{eqnarray*}
\end{lemma}
\begin{proof}[\cmit Proof] \rm We have
$$\Phi_1(n)=\sum_{\m_1=0}^{\a_1}\ldots \sum_{\m_r=0}^{\a_r}\ \frac{ \m_1(\log p_1)(\log\log p_1)+\ldots +\m_r(\log p_r)(\log\log p_r)}{ p_1^{\m_1}\ldots p_{r}^{\m_{r}}}  $$
The $i$-th term of the numerator yields the sum
$$\underbrace{\sum_{\m_1=0}^{\a_1}\ldots \sum_{\m_r=0}^{\a_r}}_{\substack{\hbox{\small the sum relative}\\ \hbox{\small to $ \m_i$ is excluded}}} \underbrace{\frac{1}{ p_1^{\m_1}\ldots p_{r}^{\m_{r}}}}_{\substack{\hbox{\small  $ p_i^{\m_i}$}\\ \hbox{\small  is excluded}}}\ \Big(\sum_{\m_i=0}^{\a_i} \frac{\m_i\big(\log p_i\big)\big(\log\log p_i\big)}{p_i^{\m_i}}\Big). $$

Consequently,
\begin{eqnarray}\label{Phi1formula}  \Phi_1(n)
  &=& \sum_{i=1}^r\underbrace{\sum_{\m_1=0}^{\a_1}\ldots \sum_{\m_r=0}^{\a_r}}_{\substack{\hbox{\small the sum relative}\\ \hbox{\small to $ \m_i$ is excluded}}} \underbrace{\frac{1}{ p_1^{\m_1}\ldots p_{r}^{\m_{r}}}}_{\substack{\hbox{\small  $ p_i^{\m_i}$}\\ \hbox{\small  is excluded}}}\ \Big(\sum_{\m_i=0}^{\a_i} \frac{\m_i\big(\log p_i\big)\big(\log\log p_i\big)}{p_i^{\m_i}}\Big)\cr &=& \sum_{i=1}^r\prod_{\substack{j =1\\ j \neq i}}^r\Big(\frac{1-p_j^{-\a_j -1}}{1-p_j^{ -1}} \Big)\Big[\sum_{\m_i=0}^{\a_i} \frac{\m_i\big(\log p_i\big)\big(\log\log p_i\big)}{p_i^{\m_i}}\Big] .
\end{eqnarray}
Now as
$$\sum_{\m=0}^{\a_i} \frac{\m}{p_i^\m}\le\sum_{j=0}^{\infty} \frac{j}{p_i^j} =\frac{1}{(p_i-1)(1-p_i^{-1})},$$
we obtain
\begin{eqnarray*} \Phi_1(n)&\le &\sum_{i=1}^r\prod_{\substack{j =1\\ j \neq i}}^r\Big(\frac{1-p_j^{-\a_j -1}}{1-p_j^{ -1}}\Big)\,.\,\frac{(\log p_i)\big(\log\log p_i\big)}{(p_i-1)(1-p_i^{-1})} \cr
 &\le &\Big(\prod_{j =1 }^r \frac{1}{1-p_j^{ -1}} \Big)\sum_{i=1}^r \frac{(\log p_i)\big(\log\log p_i\big)}{p_i-1}.
\end{eqnarray*}
\end{proof}

\begin{corollary}\label{ests1} We have the following estimate,
\begin{eqnarray*}
\limsup_{n\to \infty}\frac{\Phi_1(n)}{(\log\log n)^2(\log \log\log n)}   &\le & e^{\g}.\end{eqnarray*}
\end{corollary}

\begin{proof}[\cmit Proof]\rm Let $p(j)$ denote  the $j$-th consecutive prime number, and recall that  (\cite[(3.12-13)]{RS}, 
\begin{eqnarray}\label{p(i)est}
p(i) &\ge& \max(i \log i, 2), \qq\quad\ \  i\ge 1,
\cr p(i)&\le& i(\log i + \log\log i ), \qq \ \! i\ge 6.
\end{eqnarray}

Let  $\e>0$ and an integer  $r_0\ge 4$. If $r\le r_0$, then
\begin{eqnarray}\label{phi1.sumr1}
\sum_{i=1}^r \frac{(\log p_i)\big(\log\log p_i\big)}{p_i-1}
&\le &\d\,r_0, \qq \qq \d= \sup_{p\ge 3}\frac{(\log p)\big(\log\log p\big)}{p-1}<\infty\,.
\end{eqnarray}
If $r>r_0$, then
\begin{eqnarray*}
\sum_{i=r_0+1}^r \frac{(\log p_i)\big(\log\log p_i\big)}{p_i-1}
&\le & \Big(\max_{i>r_0}\frac{p(i)}{p(i)-1}\Big)\sum_{i=r_0+1}^r \frac{(\log p(i))\big(\log\log p(i)\big)}{p(i)}
\cr &\le & \Big(\max_{i>r_0}\frac{p(i)}{p(i)-1}\Big)\sum_{i=r_0+1}^r \frac{(\log (i\log i))\big(\log\log (i\log i)\big)}{i\log i}
\end{eqnarray*}
We choose  $r_0=r_0(\e)$ so that $\log r_0 \ge 1/\e$ and the preceding expression is bounded from above by
$$(1+\e)\sum_{i=r_0+1}^r \frac{\log\log i}{i}$$
We thus have
\begin{eqnarray}\label{phi1.sumr2}
\sum_{i=r_0+1}^r \frac{(\log p_i)\big(\log\log p_i\big)}{p_i-1}
&\le  &(1+\e)\int_{r_0}^r\frac{\log\log t}{t}\dd t
\cr &\le & (1+\e)(\log r)(\log\log r).
\end{eqnarray}
 Consequently, for some $r(\e)$,
 \begin{eqnarray}\label{phi1.sumr}
\sum_{i= 1}^r \frac{(\log p_i)\big(\log\log p_i\big)}{p_i-1}
 &\le & (1+\e)(\log r)(\log\log r), \qq r\ge r(\e).
\end{eqnarray}
By using Mertens' estimate
\begin{eqnarray}\label{prod}\prod_{p\le x}\Big(\frac{1}{1-\frac{1}{p}}\Big)=e^{\g}\log x + \mathcal O(1)\qq \quad x\ge 2,
\end{eqnarray}
 we further have
\begin{equation}\label{p(i)estappl}
 \prod_{\ell =1}^r \Big(\frac{1}{1-\frac{1}{p_\ell}} \Big)\,\le\, \prod_{\ell =1}^r \Big(\frac{1}{1-\frac{1}{p(\ell)}} \Big) \le  \prod _{p\le r(\log r +\log\log r)}\Big(\frac{1}{1- \frac{1}{p}}\Big)
\,\le \,  e^{\g} (\log r) + C\, ,
\end{equation}
if $r\ge 6$,  and so for any $r\ge 1$, modifying $C$ if necessary. As  $r=\o(n)$ and $2^{\o(n)}\le n$, we consequently have,
\begin{eqnarray*}
\Phi_1(n)
&\le & e^{\g}(1+ C\e)^2 (\log\log n)^2(\log \log\log n),\end{eqnarray*}
if $r>r_0$.
If $r\le r_0$, we have
\begin{eqnarray*}
\Phi_1(n)&\le &  \d e^{\g}(1+\e) \big((\log r_0) + C\big):=C(\e).\end{eqnarray*}
Whence,
\begin{eqnarray*}
\Phi_1(n)  &\le & e^{\g}(1+\e)^2 (\log\log n)^2(\log \log\log n)+ C(\e).\end{eqnarray*}
As $\e$ can be arbitrary small, the result follows.
\end{proof}
The following lemma is nothing but the upper bound part of  \eqref{EZ1}.
We omit the proof.
\begin{lemma} \label{tEZm}
 We have the following estimate,
\begin{eqnarray*}  \sum_{d|n} \frac{\log d}{d} &\le &\prod_{p|n}\Big(\frac{1}{1-p^{-1}}\Big) \ \sum_{p|n}\frac{\log p}{p-1}.
\end{eqnarray*}
Moreover,
\begin{eqnarray*}  \limsup_{n\to \infty}\ \frac{1}{ (\log\log n)(\log \o(n))}\sum_{d|n} \frac{\log d}{d} &\le & e^{\g}.
\end{eqnarray*}
\end{lemma}

  \subsection{\bf Lower estimates.}\label{s3}


We recall that the smallest prime divisor of an integer $n$ is noted by  $P^-(n)$.
\begin{lemma}\label{phi1min} Let  $n=p_1^{\a_1}\ldots p_r^{\a_r}$, $r\ge 1$, $\a_i\ge1$.  Then,
\begin{eqnarray*} \Phi_1(n)  &\ge  & \Big(1-\frac{1}{P^-(n)}\Big)\prod_{j =1}^r\big(1+p_j^{-1} \big)\Big[ \sum_{i=1}^r\frac{ \big(\log p_i\big)\big(\log\log p_i\big)}{p_i}\Big]\end{eqnarray*}
\end{lemma}
\begin{proof}[\cmit Proof] By \eqref{Phi1formula},
\begin{eqnarray*} \Phi_1(n) &=&\sum_{i=1}^r\prod_{\substack{j =1\\ j \neq i}}^r\Big(\frac{1-p_j^{-\a_j -1}}{1-p_j^{ -1}} \Big)\Big[\sum_{\m_i=0}^{\a_i} \frac{\m_i\big(\log p_i\big)\big(\log\log p_i\big)}{p_i^{\m_i}}\Big]
\cr &\ge &\sum_{i=1}^r\prod_{\substack{j =1\\ j \neq i}}^r\Big(\frac{1-p_j^{-\a_j -1}}{1-p_j^{ -1}} \Big)\Big[ \frac{\big(\log p_i\big)\big(\log\log p_i\big)}{p_i}\Big]
\cr &\ge  & \prod_{j =1}^r\big(1+p_j^{-1} \big)\Big[ \sum_{i=1}^r\frac{(1-p_i^{ -1})\big(\log p_i\big)\big(\log\log p_i\big)}{p_i}\Big].\end{eqnarray*}
Thus
\begin{eqnarray*} \Phi_1(n)  &\ge  &  \Big(1-\frac{1}{P^-(n)}\Big)\prod_{j =1}^r\big(1+p_j^{-1} \big)\Big[ \sum_{i=1}^r\frac{ \big(\log p_i\big)\big(\log\log p_i\big)}{p_i}\Big].\end{eqnarray*}
\end{proof}
We easily deduce from Lemma  \ref{phi1maj} and Lemma \ref{phi1min} the following corollary.
 \begin{corollary}\label{phi1est} Let $n=p_1^{\a_1}\ldots p_r^{\a_r}$, $r\ge 1$,  $\a_i\ge1$. Then,
\begin{eqnarray*}\big(1-\frac{1}{P^-(n)}\big)\prod_{j =1}^r\big(1+p_j^{-1} \big) \ \le  \frac{\Phi_1(n)}{\sum_{i=1}^r\frac{ (\log p_i)(\log\log p_i)}{p_i}}\ \le  2\,  \prod_{j =1}^r\Big(\frac{1}{1-p_j^{ -1}} \Big).\end{eqnarray*}
\end{corollary}

\begin{proposition} \label{tEZ}  
We have the following estimates
\begin{eqnarray*}\hbox{$\rm a)$}& & \limsup_{n\to \infty}\ \frac{1}{ (\log\log n)} \sum_{d|n} \frac{(\log d)}{d }\ \ge \ e^{\g}
\cr  \hbox{$\rm b)$} & &\limsup_{n\to \infty}\, \frac{\Phi_1(n)}{(\log \log n)^2(\log\log\log n)}\,\ge \,e^\g,
\cr  \hbox{$\rm c)$} & &\limsup_{n\to \infty}\, \frac{\Psi(n)}{(\log \log n)^2(\log\log\log n)}\,\ge \,e^\g.
\end{eqnarray*}
\end{proposition}
 \begin{proof}[\cmit Proof] \rm
Case a) is Erd\H os-Zaremba's lower bound of function $\Phi(n)$.   Since it is used in the proof of b) and c),   we provide a detailed  proof for the sake of completeness.
\vskip 3 pt
a) Let  $n_j=\prod_{p<e^j}p^j$. Recall that
  $p(i) \ge \max(i \log i, 2)$ if $i\ge 1$. Let   $r(j)$ be the integer defined by  the  condition $p(r(j))< e^j< p(r(j)+1)$.

 By using \eqref{formule}  and following  Gronwall's proof \cite{Gr},
 we  have,
  \begin{eqnarray*}& &  \sum_{d|n_j} \frac{\log d}{d} \ =\ \sum_{i=1}^{r(j)}\prod_{\substack{ \ell=1\\ \ell\neq i}}^{r(j)}\Big(\frac{1-p(\ell)^{-j-1}}{1-p(\ell)^{-1}}\Big)\Big[\sum_{\m=0}^{j}\frac{\m\log p(i)}{p(i)^\m}\Big]
\cr &\ge &  \frac{1}{\zeta(j+1)}\prod_{\ell=1}^{r(j)}\Big(\frac{1}{1-p(\ell)^{-1}}\Big)\sum_{i=1}^{r(j)} (1-p(i)^{-1})\frac{\log p(i)}{p(i)}\Big[1+ \frac{1}{p(i)}+\ldots +\frac{1}{p(i)^{j-1}}\Big]
\cr &= &  \frac{1}{\zeta(j+1)}\prod_{\ell=1}^{r(j)}\Big(\frac{1}{1-p(\ell)^{-1}}\Big)\sum_{i=1}^{r(j)}  \frac{\log p(i)}{p(i)}\big(1-p(i)^{-j}\big).
\end{eqnarray*}
 Recall that   $\vartheta(x)=\sum_{p\le x}\log p$ is  Chebycheff's function and that  $\vartheta(x)\ge(1-\e(x))x$, $x\ge 2$, where $\e(x)\to 0$ as $x$ tends to infinity.
Thus,
  $\log n_j = j\vartheta(e^j)= je^j(1+ o(1))$, and thus
$\log\log n_j
= j(1+ o(1))$.
\vskip 2 pt
On the one hand,   by \eqref{prod},
\begin{equation}\label{prodnj}\prod_{\ell=1}^{r(j)}\big(1-p(\ell)^{-1}\big)=
\prod_{p<e^j}\big(1-p^{-1}\big)=\frac{e^{-\g}}{j}\big(1+ \mathcal O(\frac{1}{j})\big).
\end{equation}
 And on the other, by   Mertens' estimate
 \begin{equation}\label{sumnj}\sum_{p<e^j}  \frac{\log p}{p}=j+\mathcal O(1)\ge (1+o(1)) \log \log n_j  .
\end{equation}

Thus
\begin{eqnarray} \label{lbeta1} \sum_{d|n_j} \frac{\log d}{d}  &\ge  &  (1+o(1))e^{\g}(\log\log n_j)^{2} \qq\qq j\to \infty\,,
\end{eqnarray}
since  $\zeta(j+1)\to 1$ as $j\to \infty$.

\vskip 3 pt
 b)
 Let $\s'_{-1}(n)= \sum_{d|n\,,\, d\ge 3} 1/d$. Let also  $X$  be a discrete random variable equal to $\log d$ if $d|n$ and $d\ge 3$, with probability $1/(d\s'_{-1}(n))$. By using convexity of the function  $x\log x$ on $[1,\infty)$, we get
 \begin{eqnarray*} \E X\log X&=& \sum_{\substack{d|n\\ d\ge 3}} \frac{(\log d)(\log \log d)}{d\s'_{-1}(n) }\ \ge\ (\E X)\log\,(\E X)
\cr &= & \Big(\sum_{\substack{d|n\\ d\ge 3}} \frac{(\log d)}{d\s'_{-1}(n) }\Big)\log \Big(\sum_{\substack{d|n\\ d\ge 3}} \frac{(\log d)}{d\s'_{-1}(n) }\Big)
\cr &\ge & \Big(\sum_{\substack{d|n\\ d\ge 1}} \frac{(\log d)}{d\s'_{-1}(n) }-C\Big)\Big(\log  \Big(\sum_{\substack{d|n\\ d\ge 1}} \frac{(\log d)}{d }-C\Big)-\log \s_{-1}(n)\Big) .
\end{eqnarray*}
Whence\begin{eqnarray*}  \sum_{\substack{d|n\\ d\ge 3}} \frac{(\log d)(\log \log d)}{d}  &\ge & \Big(\sum_{\substack{d|n\\ d\ge 1}} \frac{(\log d)}{d
}-C\s_{-1}(n)\Big)\Big(\log  \Big(\sum_{\substack{d|n\\ d\ge 1}} \frac{(\log d)}{d }-C\Big)
\cr & &  - \log \s_{-1}(n)\Big)
\end{eqnarray*}
Letting $n=n_j$, we deduce  from \eqref{lbeta1} that
 \begin{eqnarray*} \Psi(n) &\ge  &\sum_{\substack{d|n\\ d\ge 3}} \frac{(\log d)(\log \log d)}{d} \  \ge \ \Big((1+o(1))e^{\g}(\log\log n_j)^{2}
-C\log\log n_j\Big)
\cr & & \qq \times \Big(\log  \big\{(1+o(1))e^{\g}(\log\log n_j)^2-C\big\}
 - \log C \log\log n_j\Big)
 \cr & \ge &(1+o(1))e^{\g}(\log\log n_j)^2\log\log\log n_j.
\end{eqnarray*}
Consequently,
 \begin{eqnarray*}  \limsup_{n\to \infty}\frac{\Psi(n)}{(\log\log n)^2\log\log\log n}
 &  \ge &e^{\g}.
\end{eqnarray*}
\vskip 3 pt
 c)
 We have
\begin{eqnarray*} \Phi_1(n_j) &=&\sum_{i=1}^{r(j)}\prod_{\substack{\ell=1\\ \ell\neq i}}^{r(j)}\Big(\frac{1-p(\ell)^{-j-1}}{1-p(\ell)^{-1}}\Big)\Big[\sum_{\m=0}^{j}\frac{\m
(\log p(i))(\log\log p(i))}{p(i)^\m}\Big]
\cr &\ge &  \frac{1}{\zeta(j+1)}\prod_{\ell=1}^{r(j)}\Big(\frac{1}{1-p(\ell)^{-1}}\Big)\cr & & \quad\times\  \sum_{i=1}^{r(j)} (1-p(i)^{-1})\frac{(\log p(i))(\log\log p(i))}{p(i)}\Big[1+ \frac{1}{p(i)}+\ldots +\frac{1}{p(i)^{j-1}}\Big]
\cr &\ge  &  \frac{1}{\zeta(j+1)}(e^{\g}j)\big(1+ \mathcal O(\frac{1}{j})\big)\sum_{i=1}^{r(j)}  \frac{(\log p(i))(\log\log p(i))}{p(i)}\big(1-p(i)^{-j}\big).
\end{eqnarray*}
by \eqref{prodnj}.
Let  $0<\e <1$. By using \eqref{sumnj}, we also have  for all $j$ large enough, \begin{eqnarray*}\sum_{p<e^j}  \frac{(\log p)(\log\log p)}{p} &\ge & \sum_{e^{\e j}\le p<e^j}  \frac{(\log p)(\log\log p)}{p}
\cr &\ge & (1+o(1))\big(\log(\e j)\big)\sum_{e^{\e j}\le p<e^j}  \frac{(\log p)}{p}
\cr &\ge & (1+o(1))(1-\e)j\big(\log(\e j)\big)\big(1+ \mathcal O({1}/{j})\big)
\cr &\ge & (1+o(1))(1-\e)(\log \log n_j)\big(\log (\e \log \log n_j)\big).\end{eqnarray*}
As $\log (\e \log \log n_j) \sim \log\log \log n_j$, $j\to \infty$, we have
\begin{eqnarray*} \limsup_{j\to \infty}\frac{\Phi_1(n_j)}{(\log \log n_j)^2(\log\log \log n_j)}  &\ge  &  e^{\g}(1-\e).
\end{eqnarray*}
As $\e$ can be arbitrarily small, this proves (c).
\end{proof}

\rm
\begin{lemma}\label{Phi_2(r,n)min} We have the following estimate
\begin{eqnarray*}
\Phi_2(n)  &\ge  & (\log 2)\,\Big(\frac{P^-(n)}{P^-(n)+1}\Big) \,\Big(\prod_{i=1}^{r}\big(1+\frac{1}{ p_i}\big)\Big)\sum_{j=1}^r\big(\frac{ \log p_j}{p_j}\big).\end{eqnarray*}
\end{lemma}
\begin{proof} We observe from \eqref{Phi_2(r,n)} that
 \begin{eqnarray*}
\Phi_2(r,n)&\ge & \sum_{\m_1=0}^{\a_1}\ldots \sum_{\m_{r-1}=0}^{\a_{r-1}} \frac{1}{ p_1^{\m_1}\ldots p_{r-1}^{\m_{r-1}}}\frac{ \log p_r}{p_r}\log \Big[\sum_{i=1}^{r-1} \m_i + 1\Big].\end{eqnarray*}
It is clear that the above  multiple sum can contribute (is not null) only if  $\max_{i=1}^{r-1} \m_i \ge 1$, in which case  $\log\,  [\,\sum_{i=1}^{r-1} \m_i + 1]\ge \log 2$.
We thus have  \begin{eqnarray} \label{Phi2(r,n)min}
\Phi_2(r,n)&\ge & (\log 2)\big(\frac{ \log p_r}{p_r}\big)\,
\underbrace{
\sum_{\m_1=0}^{\a_1}\ldots \sum_{\m_{r-1}=0}^{\a_{r-1}}}_{\hbox{$\max_{i=1}^{r-1} \m_i \ge 1$}} \frac{1}{ p_1^{\m_1}\ldots p_{r-1}^{\m_{r-1}}}
\cr &= & (\log 2)\big(\frac{ \log p_r}{p_r}\big)\,\prod_{i=1}^{r-1}\Big(1+\sum_{\m_i=0}^{\a_i}\frac{1}{ p_i^{\m_i}}\Big)
 \cr &\ge  & (\log 2)\big(\frac{ \log p_r}{p_r}\big)\,\prod_{i=1}^{r-1}\Big(1+\frac{1}{ p_i}\Big).\end{eqnarray}
Consequently,
\begin{eqnarray} \label{Phi2(rn)min}
\Phi_2(n) &\ge  & (\log 2)\,\sum_{j=1}^r\big(\frac{ \log p_j}{p_j}\big)\,\prod_{\substack{i=1\\ i\neq j}}^{r}\Big(1+\frac{1}{ p_i}\Big)
\cr &\ge  & (\log 2)\,\Big(\frac{P^-(n)}{P^-(n)+1}\Big) \,\Big(\prod_{i=1}^{r}\big(1+\frac{1}{ p_i}\big)\Big)\sum_{j=1}^r\big(\frac{ \log p_j}{p_j}\big).\end{eqnarray}
\end{proof}


\section{\bf An application.} \label{s6}

We deduce from   Theorem \ref{t1} the following result.
\begin{theorem}\label{t3} Let $\eta>1$.  There exists a constant  $C(\eta)$ depending on $\eta$ only, such that for any finite set  $K$ of distinct integers, and any sequence of reals
$\{c_k, k\in K\}$, we have
\begin{eqnarray}\label{approx}\sum_{k,\ell\in K} c_kc_\ell  \frac{(k,\ell)^{2}}{k\ell}&\le & C(\eta) \sum_{\nu\in K} c_\nu^2
\,\,(\log\log\log n)^\eta\,\Psi (\nu)
.
\end{eqnarray}
Further,
 \begin{eqnarray} \label{approx1}
\sum_{k,\ell \in K}
 c_k   c_\ell\frac{(k,\ell)^{2}}{k\ell}&\le& C(\eta)\sum_{\nu \in K} c_\nu^2    (\log\log \nu)^2  (\log\log\log \nu)^{1+\eta}. \end{eqnarray}
 \end{theorem}

 This much improves  Theorem 2.5 in \cite{W1} where  a specific question related to G\'al's inequality was investigated, see \cite{W1} for details.
  The interest of inequality  \eqref{approx}, is naturally that the bound obtained
 tightly  depends on  the arithmetical structure of the support $K$ of the coefficient sequence, while being
 close to the optimal order of magnitude $(\log\log \nu)^2$.
\vskip 2 pt
Theorem \ref{t3} is obtained as a  combination of  Theorem \ref{t1} with  a  slightly more general and sharper formulation of  Theorem 2.5 in \cite{W1}.
  \begin{theorem}\label{t5}   Let
$\eta >1$. Then,  for any real $s$ such that $0<s\le 1$,
 for any sequence of reals  $\{c_k, k\in K\}$, we have
 \begin{eqnarray}\label{t1m}\sum_{k,\ell\in K} c_kc_\ell  \frac{(k,\ell)^{2s}}{k^s\ell^s}&\le & C(\eta) \sum_{\nu\in K}
c_\nu^2(\log\log\log
\nu)^\eta
\sum_{\d|\nu} \frac{(\log \d )(\log\log \d)}{\d^{2s-1}}.
\end{eqnarray}
 The constant  $C(\eta)$ depends on   $\eta$ only.
\end{theorem}

\begin{remark}\label{rems}\rm From   Theorem 2.5-(i) in \cite{W1}, follows that for every  $s>1/2$,  \begin{eqnarray}\label{i1}
 \sum_{k,\ell\in K} c_k   c_\ell\frac{(k,\ell)^{2s}}{k^s\ell^s}  &\le&\zeta(2s)  \inf_{0< \e\le 2s-1} \frac{1+\e}{\e } \, \sum_{\nu \in K}
   c_\nu^2 \, \s_{ 1+\e-2s}(\nu)   ,
 \end{eqnarray}
  $\s_{u}(\nu)$  being   the sum of $u$-th powers of  divisors of
$\nu$, for any real   $u$. As \begin{eqnarray*}
 \sum_{\d|\nu} \frac{(\log \d )(\log\log \d)}{\d^{2s-1}}\ll \sum_{\d|\nu} \frac{1}{\d^{2s-1-\e}} =\s_{ 1+\e-2s}(k)   ,
 \end{eqnarray*}
estimate \eqref{t1m} is much better than the one given \eqref{i1}.
\end{remark}
 \begin{proof}[\cmit Proof of Theorem \ref{t5}] \rm
The proof is similar to that  of Theorem 2.5 in \cite{W1} and shorter. Let  $\e>0$ and let   $J_\e$  denote the generalized Euler function.  We recall that    \begin{eqnarray}\label{jordan}
 J_\e(n)= \sum_{d|n} d^\e \m(\frac{n}{d}).
 \end{eqnarray}

We extend the sequence $\{c_k, k\in K\}$  to all $\N$ by putting $c_k= 0$ if  $k\notin K$.
 By M\"obius' formula, we have  $n^\e =\sum_{d|n} J_\e (d)$.
 By using   Cauchy-Schwarz's inequality,  we successively obtain \begin{eqnarray}  \label{HS1a}
 L&:=& \sum_{k,\ell=1}^n
 c_k   c_\ell\frac{(k,\ell)^{2s}}{k^s\ell^s}\ =\  \sum_{k,\ell \in K}  \frac{c_k   c_\ell }{k^s\ell^s}\Big\{\sum_{d\in F(K)}
J_{2s} (d)  {\bf 1}_{d|k} {\bf 1}_{d|\ell}\Big\}
\cr \hbox{($k=ud$, $\ell=vd$)}  &\le& \sum_{u,v\in F(K)} \frac{1}{u^sv^s} \Big(\sum_{d\in F(K)} \frac{J_{2s}
(d)}{d^{2s}}c_{ud}c_{vd}   \Big)
\cr &\le & \sum_{u,v\in F(K)} \frac{1}{u^sv^s} \Big(\sum_{d\in F(K)} \frac{J_{2s}
(d)}{d^{2s}}c_{ud}^2  \Big)^{1/2}\Big(\sum_{d\in F(K)} \frac{J_{2s}
(d)}{d^{2s}} c_{vd}^2   \Big)^{1/2}
\cr &=& \Big[\sum_{u \in F(K)} \frac{1}{u^s } \Big(\sum_{d\in F(K)} \frac{J_{2s}
(d)}{d^{2s}}c_{ud}^2  \Big)^{1/2}\Big]^2
\cr &\le& \Big(\sum_{u \in F(K)} \frac{1}{u^s\psi(u) } \Big)\Big(\sum_{\nu \in K} \frac{ c_\nu^2}{  \nu^{2s}   }   \sum_{\substack{u \in
F(K)\\ u|\nu} }
  J_{2s}\big( \frac{\nu}{u   }\big) u^{ s} \psi(u) \Big)  , \end{eqnarray}
where $\psi (u)>0$ is a non-decreasing function on
$\R^+$.
We then choose
$$\psi(u) = u^{-s} \psi_1(u)\sum_{t|u} t (\log t)(\log\log t),\qq \qq \psi_1(u)= (\log\log\log u)^\eta.$$
Hence,\begin{eqnarray*}
L &\le& \Big(\sum_{u \in F(K)} \frac{1}{\psi_1(u)\sum_{t|u} t (\log t)(\log\log t) } \Big)\Big(\sum_{\nu \in K} \frac{ c_\nu^2}{  \nu^{2s}   }   \sum_{\substack{u \in
F(K)\\ u|\nu} }
  J_{2s}\big( \frac{\nu}{u   }\big) \psi_1(u)\sum_{t|u} t (\log t)(\log\log t)\Big)
  \cr &\le& \Big(\sum_{u \in F(K)} \frac{1}{\psi_1(u)\sum_{t|u} t (\log t)(\log\log t) } \Big)\Big(\sum_{\nu \in K} \frac{ c_\nu^2 \psi_1(\nu)
}{  \nu^{2s}   }   \sum_{\substack{u \in
F(K)\\ u|\nu}}   J_{2s}\big( \frac{\nu}{u   }\big)\sum_{t|u} t (\log t)(\log\log t) \Big)  .
 \end{eqnarray*}
 As $\nu \in K$,  we can write
 \begin{eqnarray}\label{f}
      \sum_{\substack{u \in
F(K)\\ u|\nu }}
  J_{2s}\big( \frac{\nu}{u   }\big) \sum_{t|u} t (\log t)(\log\log t)&=&  \sum_{u|\nu}\sum_{d|\frac {\nu}u}d^{2s}\m \Big(\frac {\nu}{ud}\Big)\sum_{t|u} t (\log t)(\log\log t)
\cr & = &\sum_{d|\nu}d^{2s}\sum_{u|\frac {\nu}d}\m \Big(\frac {\nu}{ud}\Big)\sum_{t|u} t (\log t)(\log\log t)
\cr \hbox{(writing $u=tx$)}
 &=& \sum_{d|\nu}d^{2s}\sum_{t|\frac {\nu}d}t (\log t)(\log\log t)\sum_{x|\frac {\nu}{dt}}\m \Big(\frac {\nu}{dtx}\Big)
\cr \hbox{(writing $\frac {\nu}{dt}=x\theta$)}&
=& \sum_{d|\nu}d^{2s}\sum_{t|\frac {\nu}d}t (\log t)(\log\log t)\sum_{\theta|\frac {\nu}{dt}}\m (\theta)
    \cr&=&  \sum_{d|\nu}d^{2s}(\frac {\nu}d) (\log (\frac {\nu}d))(\log\log (\frac {\nu}d)),
   \end{eqnarray}
where in the last inequality we used the fact that  $\sum_{d|n}\m(d)$ equals $1$ or  $0$ according to
$n=1$ or $n>1$.

Consequently,
\begin{eqnarray*} L &\le& \Big(\sum_{u \in F(K)} \frac{1}{\psi_1(u)\sum_{t|u} t (\log t)(\log\log t) } \Big)\Big(\sum_{\nu \in K} \frac{ c_\nu^2
\psi_1(\nu) }{  \nu^{2s}   }   \sum_{d|\nu}d^{2s}(\frac {\nu}d) (\log (\frac {\nu}d))(\log\log (\frac {\nu}d)) \Big)
\cr &=& \Big(\sum_{u \in F(K)} \frac{1}{\psi_1(u)\sum_{t|u} t (\log t)(\log\log t)} \Big)\Big(\sum_{\nu \in K} c_\nu^2 \psi_1(\nu)   \sum_{\d|\nu} \frac{1}{\d^{2s}}\,\d (\log \d)(\log\log \d) \Big)  .
 \end{eqnarray*}
\vskip 7 pt
From the trivial estimate $\sum_{t|u}t (\log t)(\log\log t)\ge u (\log u)(\log\log u)$,  it follows that
 \begin{eqnarray}  \label{s}
\sum_{k,\ell=1}^n
 c_k   c_\ell\frac{(k,\ell)^{2s}}{k^s\ell^s}
  &\le& \Big(\sum_{u \ge 1 } \frac{1}{u (\log u)(\log\log u) (\log\log\log u)^\eta } \Big)
 \cr & &\times \Big(\sum_{\nu \in K} c_\nu^2   (\log\log\log \nu)^\eta\sum_{\d|\nu} \frac{ (\log \d)(\log\log \d)  }{\d^{2s-1}} \Big)
 \cr & = & C(\eta)\ \sum_{\nu \in K} c_\nu^2   (\log\log\log \nu)^\eta\sum_{\d|\nu} \frac{ (\log \d)(\log\log \d)  }{\d^{2s-1}}
  .
 \end{eqnarray}
\end{proof}
 \begin{proof}[\cmit Proof of Theorem \ref{t3}] \rm  Letting $s=1$ in Theorem \ref{t5} we get \eqref{approx}, next  using Theorem  \ref{t1} we obtain,
\begin{eqnarray}  \label{1}
\sum_{k,\ell=1}^n
 c_k   c_\ell\frac{(k,\ell)^{2}}{k\ell}
 &\le& C(\eta)\,
\sum_{\nu \in K} c_\nu^2 (\log \log \nu)^2  (\log\log\log \nu)^{1+\eta} , \end{eqnarray}
 which is \eqref{approx1}, and thus  proves Theorem \ref{t3}.
 \end{proof}

 \rm
\vskip 3 pt

 \section{\bf Concluding Remarks.} \label{s7}

\rm


The proof of Theorem \ref{t2} can be adapted with no  difficulty to similar arithmetical  functions, for instance with powers of $\log\log d$, but not to  the functions
 $S_k(n)$, $k\ge 1$,
  which  specifically depend on a derivation formula, see  after  \eqref{sisu2}.
    We   remark that a simple convexity argument shows that
  \begin{eqnarray}\label{Phietamin}\limsup_{n\to \infty}\ \frac{S_k(n)}{ (\log\log n)^{1+k}} &\ge&e^{\g} .
\end{eqnarray}
Let  indeed $X$ be a discrete random variable equal to $\log d$ if $d|n$,  with probability  $1/(d\s_{-1}(n))$. Then,
$$\E X^k=  \sum_{d|n} \frac{(\log d)^k}{d\s_{-1}(n)}\ge (\E X)^k=  \big(\sum_{d|n} \frac{\log d}{d\s_{-1}(n)}\big)^k.$$
Whence,
$$ S_k(n)=\sum_{d|n} \frac{(\log d)^k}{d }\ge \s_{-1}(n)^{1-k}\big(\sum_{d|n} \frac{\log d}{d}\big)^k.$$
As $\s_{-1}(n)\le (1+o(1))e^\g\log\log n$,
by using \eqref{lbeta1}  we deduce that
\begin{eqnarray*} S_k(n_j)&\ge& (1+o(1))e^{(1-k)\g}(\log\log n_j)^{1-k}\big(e^\g(\log\log n_j)^2)^k\cr &= &(1+o(1))e^{\g}(\log\log n_j)^{1+k}.
\end{eqnarray*}

 Moreover for integers $n$ having  sufficiently spaced prime divisors, this lower bound is optimal.
More precisely, there exists a constant $C(k)$ depending on $k$ only, such that for any integer $n=\prod_{i=1}^r p_i^{\a_i}$ satisfying the condition
$\sum_{i=1}^r\frac{1}{p_i-1}<2^{1-k},
$
one has
\begin{eqnarray}\label{Phietaminmajex}S_k(n)\ \le \ C(k)(\log\log n)^{k} \s_{-1}(n) .
\end{eqnarray}
 As $\s_{-1}(n)\le C\log\log n$, it follows that
$S_k(n)\ \le \ C(\eta)(\log\log n)^{1+\eta}$.
 \vskip 7 pt

 \vskip 3pt \hskip -2pt We conclude with some remarks concerning  Davenport's  function $w(n)$. At first,  if $p_1,\ldots,p_r$ are the  $r$ first consecutive prime numbers and $n=p_1 \ldots p_r$,
then  $w( n)\sim\log\,\o(n)$.
Next,  the obvious bound $w( n)\ll\log\log\log n$ holds true when the prime divisors of  $n$ are large, for instance when  for a  given positive number $B$, these prime divisors, write them $p_1,\ldots, p_r$,  satisfy
 \begin{eqnarray}\label{prop.pfinite}  \sum_{j=1}^r\frac{\log p_{j}}{p_{j}} \le B \qq \hbox{  and} \qq p_1\ldots p_r\gg e^{e^B}.
  \end{eqnarray}
 More generally, one can establish the following result.
Let $\{p_i, i\ge 1\}$ be an increasing  sequence of prime numbers enjoying the following property
\begin{eqnarray}\label{prop.p} p_1\ldots  p_s&\le  & p_{s+1}\qq\qq s=1,2,\ldots\, .
  \end{eqnarray}

Numbers of the form   $n=p_1\ldots  p_\nu$ with $p_1\ldots  p_{i-1}\le    p_i$, $2\le i\le \nu$, $\nu=1,2,\ldots$ appear as extremal numbers in some divisors questions, see Erd\H os and Hall   \cite{EH}.

 \begin{lemma}\label{b(n)}Let  $\{p_i, i\ge 1\}$ be an increasing sequence of prime numbers satisfying condition \eqref{prop.p}.
There exists a constant $C$, such that if $p_1\ge C$, then for any integer  $n= p_1^{\a_1}\ldots p_r^{\a_r}$ such that $\a_i\ge 1$ for each  $i$,
we have
$w( n)\le \log\log\log n$.
\end{lemma}
\begin{proof}[{\cmit Proof.}]
\rm We use the following inequality. Let $0<\theta<1$.  There exists a number $h_\theta$  such that for any $h\ge h_\theta$ and any $H$ such that $e^{\frac{\theta}{(1-\theta)\log 2}}\le H\le h$,
we have  \begin{eqnarray}\label{hH} h&\le  &e^h\, \log \frac{\log(H+h)}{\log H}\, .
  \end{eqnarray}
Indeed,   note that $\log (1+x) \ge \theta x$ if $0\le x \le (1-\theta)/\theta$. Let $h_\theta$ be such that if $h\ge h_\theta$, then $h\log h \le \theta(\log 2) e^h$. Thus
 \begin{eqnarray*}h&  \le & e^h\,\theta\frac{\log 2}{\log h}\le e^h\, \theta\frac{\log 2}{\log H} \le e^h\,\log \Big(1+\frac{\log 2}{\log H}\Big)=e^h\,\log \Big(\frac{\log 2H}{\log H}\Big)\cr&\le& e^h\,\log \Big(\frac{\log H+h}{\log H}\Big) \, .
  \end{eqnarray*}

We shall show by a recurrence on $r$ that
\begin{eqnarray}\label{lll} \sum_{i=1}^r\frac{\log p_i}{p_i}&\le  & \log\log\log (p_1\ldots p_r)\, .
  \end{eqnarray}
This is trivially true if $r=1$ by the notation made in the Introduction, and since  $p\ge 2$. Assume that \eqref{lll} is fulfilled for $s=1, \ldots , r-1$. Then, by the recurrence assumption,
\begin{eqnarray*} \sum_{i=1}^r\frac{\log p_i}{p_i}&\le  & \log\log\log (p_1\ldots p_{r-1} )+ \frac{\log p_r}{p_r}\, .
  \end{eqnarray*}
Put $H=\sum_{i=1}^{r-1}\log p_i$, $h=\log p_r$. It suffices to show that\begin{eqnarray*} \frac{\log p_r}{p_r}\ =\ \frac{h}{e^h} &\le  & \log\frac{\log \sum_{i=1}^r\log p_i}{\log \sum_{i=1}^{r-1}\log p_i}\ =\ \, \log\frac{\log H+h}{\log H},
  \end{eqnarray*}
  But $H\le h$, by assumption \eqref{prop.p}. Choose $C=e^{\frac{\theta}{(1-\theta)\log 2}}$. Then $H\ge \log p_1\ge e^{\frac{\theta}{(1-\theta)\log 2}}$. The searched inequality thus follows from  \eqref{hH}.

  Let $n= p_1^{\a_1}\ldots p_r^{\a_r}$, where $\a_i\ge 1$ for each $i$. We have  $w( n)\le \log\log\log (p_1\ldots p_r)\le \log\log\log n$.
\end{proof}
\vskip 6 pt


\bigskip

\noindent {\bf Acknowledgements.}
 The author is much grateful to
 an anonymous referee for having let him know the  papers of  Sitaramaiah and Subbarao \cite{SS2,SS1}, and for useful remarks. The author also thank a second referee  for useful remarks.

 \end{document}